\documentclass[11pt,reqno]{amsart}
\usepackage{graphicx}
\usepackage{hyperref}
\usepackage{amssymb}
\usepackage{cite}
\usepackage{amsmath}
\usepackage{latexsym}
\usepackage{amscd}
\usepackage{amsthm}
\usepackage{mathrsfs}
\usepackage{url}

\newtheorem{thm}{Theorem}
\newtheorem{corr}[thm]{Corollary}
\newtheorem{lem}[thm]{Lemma}

\newtheorem{exam}{Example}[section]

\theoremstyle{definition}
\newtheorem{defn}{Definition}[section]
\theoremstyle{remark}
\newtheorem{rem}[thm]{Remark}

\def\C{\mathbb C}
\def\R{\mathbb R}

\def\N{\mathbb N}

\def\cal{\mathcal}
\def\f{\frac}

\def\pt{\partial}
\def\n{\nabla}
\def\e{\varepsilon}

\begin{document}
\title[Gradient estimates for elliptic equations on manifolds]{Gradient estimates via two-point functions for elliptic equations on manifolds}
\author{Ben~Andrews}
\address{Mathematical Sciences Institute, Australian National University, ACT 2601, Australia}
\email{Ben.Andrews@anu.edu.au}
\author{Changwei~Xiong}
\address{Mathematical Sciences Institute, Australian National University, Canberra ACT 2601, Australia}
\email{changwei.xiong@anu.edu.au}
\date{\today}
\thanks{This research was partly supported by Discovery Projects grant DP120102462 and Australian Laureate Fellowship FL150100126 of the Australian Research Council. }

\maketitle

\begin{abstract}
We derive estimates relating the values of a solution at any two points to the distance between the points, for quasilinear isotropic elliptic equations on compact Riemannian manifolds, depending only on dimension and a lower bound for the Ricci curvature. These estimates imply sharp gradient bounds relating the gradient of an arbitrary solution at given height to that of a symmetric solution on a warped product model space.  We also discuss the problem on Finsler manifolds with nonnegative weighted Ricci curvature, and on complete manifolds with bounded geometry, including solutions on manifolds with boundary with Dirichlet boundary condition.  Particular cases of our results include gradient estimates of Modica type.
\end{abstract}

\section{Introduction}

Let $(M^n,g)$ be a compact Riemannian manifold without boundary. Assume the Ricci curvature of $M$ has lower bound $Ric\geq (n-1)\kappa $. We consider ``isotropic'' equations of the following form:
\begin{equation}\label{eq-G}
\left[\alpha(u,|Du|)\frac{D_iuD_ju}{|Du|^2}+\beta(u,|Du|)(\delta_{ij}-\frac{D_iuD_ju}{|Du|^2})\right]D_iD_ju+q(u,|Du|)=0.
\end{equation}
We assume that Equation \eqref{eq-G} is nonsingular, i.e. the left hand side of \eqref{eq-G} is continuous on $\R\times TM\times Sym^2(T^*M)$, $\alpha$ and $\beta$ are nonnegative functions, and that $\beta(s,t)>0$ for $t>0$. Our first main result is the following estimate:
\begin{thm}\label{thm1}
Let $(M^n,g)$ be a compact Riemannian manifold with non-negative Ricci curvature, and let $u$ be a viscosity solution of Equation \eqref{eq-G}. Suppose $\varphi:[a,b]\rightarrow [\inf u,\sup u]$ is a $C^2$ solution of
\begin{align}
\alpha(\varphi,\varphi')\varphi''+q(\varphi,\varphi')&=0\quad\text{on}\ [a,b];\label{cond1.2}\\
\varphi(a)=\inf u;\quad \varphi(b)=\sup u;\quad \varphi'&>0\quad\text{on}\ [a,b].\label{cond1.1}
\end{align}
Moreover let $\psi$ be the inverse of $\varphi$, i.e. $\psi(\varphi(z))=z$. Then we have
\begin{equation*}
\psi(u(y))-\psi(u(x))-d(x,y)\leq 0,\: \forall x,y\in M.
\end{equation*}
\end{thm}

By allowing $y$ to approach $x$, we deduce the following gradient estimate:

\begin{corr}\label{cor:gradest}
Under the assumptions of Theorem \ref{thm1}, $|Du(x)|\leq \varphi'(\psi(u(x)))$ for all $x\in M$.
\end{corr}

By applying this result in special situations, we recover several previously known results, known as Modica-type gradient estimates.  These originate from the work of Modica \cite{Mod85}, who considered bounded solutions on $\R^n$ of the equation
\begin{equation}\label{eq:Modica}
\Delta u-Q'(u)=0,
\end{equation}
where $Q$ is chosen (by adding a suitable constant) to be non-negative on the range of $u$.  The key gradient estimate of \cite{Mod85} is the following: $|Du(x)|^2\leq 2Q(u(x))$ for all $x$.  This was proved by differentiating Equation \eqref{eq:Modica} to derive a maximum principle for the function $P = \frac12|Du|^2-Q(u)$, and subsequent works have followed this method (sometimes called the ``$P$-function method'').  We observe that this is a consequence of Corollary~\ref{cor:gradest}:  In this case $\alpha=1$, and Equation \eqref{cond1.2} is equivalent to the statement that $P=\frac12(\varphi')^2-Q(\varphi)$ is constant. We can define $\varphi$ by solving this with $P=0$, and then we have $\varphi'(\psi(z)) = \sqrt{2Q(z)}$ for each $z\in[\inf u,\sup u]$, so that the estimate of Corollary \ref{cor:gradest} is exactly Modica's estimate.

Caffarelli, Garofalo and Segala \cite{CGS94} generalized Modica's result to critical points of energies of the form
\begin{equation}
\cal{E}(u)=\int_{\R^n} \left(\f{1}{2}\Phi(|D u|^2)+Q(u)\right)dx,
\end{equation}
where $\Phi \in C^3(\R^+)$ with $\Phi(0)=0$ and $Q\in C^2(\R)$ (Modica's result corresponds to $\varPhi(z)=z$). The Euler-Lagrange equation for $\cal{E}$ is given by
\begin{equation}\label{eq1}
div(\Phi'(|D u|^2) D u)-Q'(u)=0.
\end{equation}
Note that Equation \eqref{eq1} is a special case of \eqref{eq-G}, with $\alpha = 2\varPhi''(z)z+\varPhi'(z)$ and $\beta=\varPhi'(z)$ where $z=|Du|^2$.  In \cite{CGS94} the following estimate was derived:  $P\leq 0$, where
\begin{equation*}
P=\Phi'(|D u|^2)|D u|^2-\f{1}{2}\Phi(|D u|^2)-Q(u)
\end{equation*}
(again $Q$ is chosen to be non-negative on the range of $u$).  As before, this estimate is a direct consequence of our Corollary \ref{cor:gradest}:  In this case Equation~\eqref{cond1.2} becomes
$$
0 = (2\varPhi''((\varphi')^2)(\varphi')^2+\varPhi'((\varphi')^2))\varphi'' - Q'(\varphi).
$$
Multiplying by $\varphi'$ we obtain
$$
0= \left(\varPhi'((\varphi')^2)(\varphi')^2-\frac12\varPhi((\varphi')^2)-Q(\varphi)\right)'=P'.
$$
Defining $\varphi$ by solving $P=0$, we obtain a solution of \eqref{cond1.2} and deduce the claimed inequality from Corollary \ref{cor:gradest}.

Subsequent to \cite{Mod85} and \cite{CGS94}, many other authors have considered related problems in $\R^n$, e.g. \cite{DG02,FV13,FV14,CFV14,FV10,CFV12} and on Riemannian manifolds with non-negative Ricci curvature \cite{FV11,FSV13,MR10,MW14}. The proofs of these results involve a $P$-function constructed from the solution $u$ and the first derivatives $Du$.
The gradient estimates amount to pointwise inequalities on the $P$-function, deduced by application of the maximum principle to an equation resulting from differentiation of the equation satisfied by $u$.  

In this paper we use a different approach, deriving the two-point estimate of Theorem \ref{thm1} and then deducing the gradient estimate of Corollary \ref{cor:gradest}.  We refer the reader to the papers \cite{AC09,AC13} and the recent survey \cite{And14} which gives a discussion of the application of two-point estimates in a variety of geometric contexts.  This approach has several advantages:  First, the estimate is easily motivated, since it gives a comparison between an arbitrary solution and a particular symmetric solution on a product manifold.  This implies in particular that the resulting estimate is sharp.  Second, the details of the proof are comparatively simple and geometric compared to the calculations involved in the $P$-function approach.  For example, this increased simplicity allows us to treat rather arbitrary isotropic equations, and not only those which arise as Euler-Lagrange equations for variational problems, which was the case for the results obtained using the $P$-function method:  It is clear from our argument that although the variational or divergence structure seemed important in the $P$-function computations, it is in fact irrelevant to the validity of the estimate.
Finally, our argument does not involve differentiating the equation, and consequently applies (using ideas from \cite{Li16,LW16}) with minimal regularity requirements on the solution $u$, corresponding to the viscosity solution requirement in Theorem \ref{thm1}.   Throughout the paper we use the terminology of viscosity solutions from \cite{CIL92}.

We now describe the additional results we obtain using this method:

The simplicity of our method allows us to extend the proof to a more general situation of manifolds with a negative lower Ricci curvature bound, again by  comparison with a suitable one-dimensional ``warped product'' solution:

\begin{thm}\label{thm2}
Let $(M,g)$ be a compact Riemannian manifold (possibly with boundary, in which case we assume the boundary is locally convex and impose the Neumann boundary condition), and $\kappa< 0$ such that $\text{\rm Ric}\geq (n-1)\kappa g$.  Let $\bar M = N\times[a,b]$ and $\bar g = ds^2 + \rho(s)^2g^N$ be such that $\text{\rm Ric}(\partial_s,\partial_s)=(n-1)\kappa$ and $\rho'/\rho$ is strictly increasing, and let $\bar u(x,s) = \varphi(s)$ be a solution of \eqref{eq-G} on $\bar M$, where $\varphi$ is an increasing $C^2$ diffeomorphism from $[a,b]$ to $[c,d]$.  Let $\psi$ be the inverse function of $\varphi$.  Let $u$ be a viscosity solution of \eqref{eq-G} on $M$ with range contained in $[c,d]$.  Then for all $x$ and $y$ in $M$ we have
$$
\psi(u(y))-\psi(u(x))-d^M(x,y)\leq 0.
$$
\end{thm}

As we explain in Section \ref{sec4}, the assumption that $\bar u$ is a solution of \eqref{eq-G} is equivalent to a certain elliptic equation for $\varphi$ which involves $\alpha$ and $\beta$ and also the warping factor $\rho$.


\begin{corr}\label{cor:gradest2}
Under the assumptions of Theorem \ref{thm2}, $|Du(x)|\leq \varphi'(\psi(u(x)))$ for all $x\in M$.
\end{corr}


It is an interesting question whether such a result holds also in this generality in the case $\kappa>0$.  Our proof does not appear to apply in that situation.  However we discuss a different argument which works for general $\kappa\in \R $ under somewhat more stringent assumptions, in Section \ref{sec5}.

In addition we notice that the method of two-point functions also applies for equations on Finsler measure spaces with nonnegative weighted Ricci curvature, where we only consider the variational equations corresponding to \eqref{eq1}. We address this problem in Section \ref{sec6}. Then in Section~\ref{sec7}, we discuss the resulting gradient estimates of Modica type and related rigidity results.

In the final section (Section~\ref{sec8}) we present various extensions of the two-point functions method. For complete noncompact Riemannian manifolds with bounded geometry, the ``translation invariance'' of Equation~\eqref{eq-G} plays a key role; while for compact Riemannian or Finsler manifolds with boundary, the estimates near the boundary for solutions with Dirichlet boundary condition are crucial. Besides, we also consider anisotropic PDEs on certain possibly unbounded domains with boundary in $\R^n$, a case where the $P$-function method seems hard to apply. See Section~\ref{sec8} for the precise statements.

The paper is built up as follows. In Section \ref{sec2} we recall background material including the definition of viscosity solutions, maximum principles for semicontinuous functions on manifolds, and the first and second variation formulas for the arclength of a curve. In Section \ref{sec3} we give the proof of Theorem \ref{thm1} and Corollary \ref{cor:gradest}.  The more general result of Theorem \ref{thm2} is proved in Section \ref{sec4}. Section \ref{sec5} discusses a different argument which applies for general lower bounds on the Ricci curvature. Section \ref{sec6} is devoted to the setting of compact Finsler manifolds with nonnegative weighted Ricci curvature. Section~\ref{sec7} gives the pointwise gradient estimates of Modica type and some rigidity results. The final section collects extensions as stated above.

\section{Preliminaries}\label{sec2}

\subsection{Definition of viscosity solutions on manifolds}

Let $M$ be a Riemannian manifold. We use the following notations:
\begin{align*}
USC(M)&=\{u:M\rightarrow \R| u\text{ is upper semicontinuous}\},\\
LSC(M)&=\{u:M\rightarrow \R| u\text{ is lower semicontinuous}\}.
\end{align*}
Next we introduce the semijets on manifolds.
\begin{defn}
For a function $u\in USC(M)$, the second order superjet of $u$ at a point $x_0\in M$ is defined by
\begin{align*}
\mathcal{J}^{2,+}u(x_0)&:=\left\{(D\varphi(x_0),D^2\varphi(x_0)):\varphi\in C^2(M),\text{ such that }u-\varphi\right.\\
                       &\qquad\left.\text{attains a local maximum at }x_0\right\}.
\end{align*}
For $u\in LSC(M)$, the second order subjet of $u$ at $x_0\in M$ is defined by
\begin{align*}
\mathcal{J}^{2,-}u(x_0)&:=-\mathcal{J}^{2,+}(-u)(x_0).
\end{align*}
\end{defn}

We also define the closures of $\mathcal{J}^{2,+}u(x_0)$ and $\mathcal{J}^{2,-}u(x_0)$ by
\begin{align*}
\bar{\mathcal{J}}^{2,+}u(x_0)&=\{(p,X)\in T_{x_0}M\times Sym^2(T^*_{x_0}M)|\text{there is a sequence }(x_j,p_j,X_j)\\
                             &\qquad\text{such that }(p_j,X_j)\in \mathcal{J}^{2,+}u(x_j)\\
                             &\qquad\text{and }(x_j,u(x_j),p_j,X_j)\rightarrow (x_0,u(x_0),p,X)\text{ as }j\rightarrow \infty\}.\\
\bar{\mathcal{J}}^{2,-}u(x_0)&=-\bar{\mathcal{J}}^{2,+}(-u)(x_0).
\end{align*}
Now we can define the viscosity solution for the general equation
\begin{equation}\label{eq-F}
F(x,u,Du,D^2u)=0
\end{equation}
on $M$: Assume $F\in C(M\times \R\times TM\times Sym^2(T^*M))$ is degenerate elliptic, i.e.
\begin{equation*}
F(x,r,p,X)\leq F(x,r,p,Y),\text{ whenever }X\leq Y.
\end{equation*}
\begin{defn}\begin{enumerate}
  \item A function $u\in USC(M)$ is a viscosity subsolution of \eqref{eq-F} if for all $x\in M$ and $(p,X)\in \mathcal{J}^{2,+}u(x)$,
\begin{equation*}
F(x,u(x),p,X)\geq 0.
\end{equation*}
  \item A function $u\in LSC(M)$ is a viscosity supersolution of \eqref{eq-F} if for all $x\in M$ and $(p,X)\in \mathcal{J}^{2,-}u(x)$,
\begin{equation*}
F(x,u(x),p,X)\leq 0.
\end{equation*}
  \item A viscosity solution of \eqref{eq-F} is a continuous function which is both a viscosity subsolution and a viscosity supersolution of \eqref{eq-F}.
\end{enumerate}
\end{defn}

\subsection{Maximum principle for semicontinuous functions}

\begin{thm}[Theorem 3.2 in \cite{CIL92}]
Let $M_1^{n_1}$,\dots,$M_k^{n_k}$ be Riemannian manifolds, and $\Omega_i\subset M_i$ open subsets. Let $u_i\in USC(\Omega_i)$ and $\varphi\in C^2(\Omega_1\times\dots\times \Omega_k)$. Suppose the function
\begin{equation*}
w(x_1,\dots,x_k):=u_1(x_1)+\dots+u_k(x_k)-\varphi(x_1,\dots,x_k)
\end{equation*}
attains a maximum at $(\hat{x}_1,\dots,\hat{x}_k)$ on $\Omega_1\times\dots\times \Omega_k$. Then for each $\lambda>0$ there exists $X_i\in Sym^2(T^*_{\hat{x}_i}M_i)$ such that
\begin{equation*}
(D_{x_i}\varphi(\hat{x}_1,\dots,\hat{x}_k),X_i)\in \bar{\mathcal{J}}^{2,+}u_i(\hat{x}_i)\text{ for }i=1,\dots,k,
\end{equation*}
and the block diagonal matrix with entries $X_i$ satisfies
\begin{equation*}
-(\frac{1}{\lambda}+||A||)I\leq \begin{pmatrix}X_1&\cdots&0\\ \vdots&\ddots&\vdots\\0&\cdots&X_k\end{pmatrix}\leq A+\lambda A^2,
\end{equation*}
where $A=D^2\varphi(\hat{x}_1,\dots,\hat{x}_k)$.
\end{thm}

\subsection{First and second variation formulae for arclength}

Let $\gamma_0:[0,l]\rightarrow M$ be a geodesic in $M$ parametrised by arc length. Suppose $\gamma(\e,s)$ is any smooth variation of $\gamma_0(s)$ with $\e\in (-\e_0,\e_0)$. Then the first variation formula for arclength is
\begin{equation*}
\f{\pt}{\pt \e}\bigg|_{\e=0}L(\gamma(\e,\cdot))=\langle \gamma_s,\gamma_\e\rangle\big|_0^l,
\end{equation*}
where $\gamma_s$ is the unit tangent vector of $\gamma_0$ and $\gamma_\e=\f{\pt}{\pt \e}\gamma$ is the variational vector field.

Furthermore, the second variation formula is given by
\begin{equation*}
\f{\pt^2}{\pt \e^2}\bigg|_{\e=0}L(\gamma(\e,\cdot))=\int_0^l\left( |\n_{\gamma_s} (\gamma_\e^\perp)|^2-R(\gamma_s,\gamma_\e,\gamma_\e,\gamma_s)\right)ds+\langle \gamma_s,\n_{\gamma_\e} \gamma_\e \rangle\big|_0^l,
\end{equation*}
where $\gamma_\e^\perp$ means the normal part of the variational vector. Here and in the sequel we use the convention on the Riemannian curvature tensor $R$ such that $Ric(X,Y)=tr_g(R(X,\cdot,\cdot,Y))$ for $X,Y\in T_x M$.

\section{Riemannian manifolds with nonnegative Ricci curvature}\label{sec3}

First we prove the following modulus of continuity estimate, which implies Theorem \ref{thm1} immediately.
\begin{thm}\label{thm3.1}
Let $(M^n,g)$ be a compact Riemannian manifold with $Ric\geq 0$ and $u$ be a viscosity solution of Equation \eqref{eq-G}. Suppose the barrier $\varphi:[a,b]\rightarrow [\inf u,\sup u]$ satisfies
\begin{align}
&\qquad\qquad\varphi'>0,\label{cond3.1}\\
\f{d}{dz}&\left(\frac{q(\varphi,\varphi')+\varphi''\alpha(\varphi,\varphi')}{\varphi'\beta(\varphi,\varphi')}\right)<0.\label{cond3.2}
\end{align}
Moreover let $\psi$ be the inverse of $\varphi$, i.e. $\psi(\varphi(z))=z$. Then we have
\begin{equation}\label{eq3.1}
\psi(u(y))-\psi(u(x))-d(x,y)\leq 0,\: \forall x,y\in M.
\end{equation}
\end{thm}
By letting $y$ approach $x$, we get the following gradient bound.
\begin{corr}\label{corr3.1}
Under the conditions of Theorem \ref{thm3.1}, if moreover $u\in C^1(M)$, then for every $x\in M$ we have
\begin{equation*}
|\n u(x)|\leq \varphi'(\psi(u(x))).
\end{equation*}
\end{corr}
Now we show how Theorem \ref{thm3.1} implies Theorem \ref{thm1}. Let $\varphi$ satisfy \eqref{cond1.2} and \eqref{cond1.1} in Theorem \ref{thm1}. Then for sufficiently small $\delta>0$, we can solve
\begin{align*}
&\alpha(\varphi_\delta,\varphi_\delta')\varphi_\delta''+q(\varphi_\delta,\varphi_\delta')=-\delta z\cdot \varphi_\delta'\cdot \beta(\varphi_\delta,\varphi_\delta'),\\
&\qquad \qquad\varphi_\delta(a)=\varphi(a),\quad \varphi_\delta'(a)=\varphi'(a),
\end{align*}
to get $\varphi_\delta$ which satisfies \eqref{cond3.1} and \eqref{cond3.2}. So by Theorem \ref{thm3.1} we have \eqref{eq3.1} for $\varphi_\delta$. Letting $\delta\rightarrow 0^+$, we finish the proof of Theorem \ref{thm1}.

Next we focus on proving Theorem \ref{thm3.1}. For that purpose we need a lemma about the behaviour of semijets when composed with an increasing function.

\begin{lem}\label{lem:incjet}
Let $u$ be a continuous function. Let $\varphi:\R\rightarrow \R$ be a $C^2$ function with $\varphi'>0$. Let $\psi$ be the inverse of $\varphi$, so that
\begin{equation*}
\varphi(\psi(u(x)))=u(x).
\end{equation*}
\begin{enumerate}
  \item Suppose $(p,X)\in \mathcal{J}^{2,+}(\psi\circ u)(x_0)$. Then
  \begin{equation*}
  (\varphi'p,\varphi''p\otimes p+\varphi'X)\in \mathcal{J}^{2,+}u(x_0),
  \end{equation*}
  where all derivatives of $\varphi$ are evaluated at $\psi\circ u(x_0)$.
  \item Suppose $(p,X)\in \mathcal{J}^{2,-}(\psi\circ u)(x_0)$. Then
  \begin{equation*}
  (\varphi'p,\varphi''p\otimes p+\varphi'X)\in \mathcal{J}^{2,-}u(x_0),
  \end{equation*}
  where all derivatives of $\varphi$ are evaluated at $\psi\circ u(x_0)$.
  \item The same holds if we replace the semijets by their closures.
\end{enumerate}
\end{lem}
\begin{proof}
(1) Recall the definition of the superjet:
\begin{align*}
\mathcal{J}^{2,+}u(x_0)&:=\{(D\varphi(x_0),D^2\varphi(x_0)):\varphi\in C^2(M),\text{ such that }u-\varphi\\
                       &\text{attains a local maximum at }x_0\}.
\end{align*}
Assume $(p,X)\in \mathcal{J}^{2,+}(\psi\circ u)(x_0)$. Let $h$ be a $C^2$ function such that $\psi(u(x))-h(x)$ has a local maximum at $x_0$ and $(Dh,D^2h)(x_0)=(p,X)$. Since $\varphi$ is increasing, we know $u(x)-\varphi(h(x))=\varphi(\psi(u(x)))-\varphi(h(x))$ has a local maximum at $x_0$. So it follows that
\begin{equation*}
(\varphi'p,\varphi''p\otimes p+\varphi'X)\in \mathcal{J}^{2,+}u(x_0).
\end{equation*}

(2) is similar. (3) follows by approximation.
\end{proof}

\begin{proof}[Proof of Theorem \ref{thm3.1}]
The proof is by contradiction. Assume there exists some $\varepsilon_0>0$ such that
\begin{equation*}
\psi(u(y))-\psi(u(x))-d(x,y)\leq \varepsilon_0,
\end{equation*}
for any $x,y\in M$ and with equality for some $x_0\neq y_0$.

Next we want to replace $d(x,y)$ by a smooth function $\tilde{d}(x,y)$ on a neighbourhood of $(x_0,y_0)$.  To construct this we let $\gamma_0$ be the unit speed length-minimizing geodesic joining $x_0$ and $y_0$, with length $l=L(\gamma_0)$. Let $\{e_i(s)\}_{i=1}^n$ be parallel orthonormal vector fields along $\gamma_0$ with $e_n(s)=\gamma_0'(s)$ for each $s$.  Then in small neighbourhoods $U_{x_0}$ about $x_0$ and $U_{y_0}$ about $y_0$, there are mappings $x\mapsto (a_1(x),\dots,a_n(x))$ and $y\mapsto (b_1(y),\dots,b_n(y))$ defined by
\begin{equation*}
x=\exp_{x_0}(\sum_1^n a_i(x)e_i(0)),\quad y=\exp_{y_0}(\sum_1^n b_i(y)e_i(l)).
\end{equation*}
Then for some $C^2$ nonnegative function $f:[0,l]\rightarrow \R$ to be determined, we can define a smooth function $\tilde{d}(x,y)$ in $U_{x_0}\times U_{y_0}$ to be the length of the curve
\begin{equation*}
\exp_{\gamma_0(s)}\left(\frac{l-s}{l}\sum_1^na_i(x)\frac{f(s)}{f(0)}e_i(s)+\frac{s}{l}\sum_1^nb_i(y)\frac{f(s)}{f(l)}e_i(s)\right),\ s\in [0,l].
\end{equation*}
It is easy to see that $d(x,y)\leq \tilde{d}(x,y)$ in $U_{x_0}\times U_{y_0}$ and with equality at $(x_0,y_0)$. Therefore we have
\begin{equation*}
\psi(u(y))-\psi(u(x))-\tilde{d}(x,y)\leq \varepsilon_0,
\end{equation*}
for any $(x,y)\in U_{x_0}\times U_{y_0}$ and with equality at $(x_0,y_0)$.

Thus we can apply the maximum principle to conclude that for each $\lambda>0$, there exist $X\in Sym^2(T^*_{x_0}M)$ and $Y\in Sym^2(T^*_{y_0}M)$ such that
\begin{align*}
(D_y\tilde{d}(x_0,y_0),Y)&\in \bar{\mathcal{J}}^{2,+}(\psi\circ u)(y_0),\\
(D_x\tilde{d}(x_0,y_0),-X)&\in \bar{\mathcal{J}}^{2,+}(-\psi\circ u)(x_0),\\
i.e.\ (-D_x\tilde{d}(x_0,y_0),X)&\in \bar{\mathcal{J}}^{2,-}(\psi\circ u)(x_0),
\end{align*}
and
\begin{equation*}
\begin{pmatrix}-X&0\\0&Y\end{pmatrix}\leq M+\lambda M^2,
\end{equation*}
where $M=D^2\tilde{d}(x_0,y_0)$.

Note that $D_y\tilde{d}(x_0,y_0)=e_n(l)$ and $D_x\tilde{d}(x_0,y_0)=-e_n(0)$. By Lemma \ref{lem:incjet}, we have
\begin{align*}
(\varphi'(z_{y_0})e_n(l),\varphi'(z_{y_0})Y+\varphi''(z_{y_0})e_n(l)\otimes e_n(l))&\in \bar{\mathcal{J}}^{2,+}u(y_0),\\
(\varphi'(z_{x_0})e_n(0),\varphi'(z_{x_0})X+\varphi''(z_{x_0})e_n(0)\otimes e_n(0))&\in \bar{\mathcal{J}}^{2,-}u(x_0),
\end{align*}
where $z_{x_0}=\psi(u(x_0))$ and $z_{y_0}=\psi(u(y_0))$.

On the other hand, since $u$ is both a subsolution and a supersolution of \eqref{eq-G}, we have
\begin{equation*}
tr(\varphi'(z_{y_0})A_2 Y+\varphi''(z_{y_0})A_2e_n(l)\otimes e_n(l))+q(\varphi'(z_{y_0}),\varphi(z_{y_0}))\geq 0
\end{equation*}
and
\begin{equation*}
tr(\varphi'(z_{x_0})A_1 X+\varphi''(z_{x_0})A_1e_n(0)\otimes e_n(0))+q(\varphi'(z_{x_0}),\varphi(z_{x_0}))\leq 0,
\end{equation*}
where
\begin{equation*}
A_1=\begin{pmatrix}\beta(\varphi(z_{x_0}),\varphi'(z_{x_0}))&\dots&0&0\\ \vdots&\ddots&\vdots&\vdots\\0&\dots&\beta(\varphi(z_{x_0}),\varphi'(z_{x_0}))&0\\0&\dots&0&\alpha(\varphi(z_{x_0}),\varphi'(z_{x_0}))\end{pmatrix},
\end{equation*}
and
\begin{equation*}
A_2=\begin{pmatrix}\beta(\varphi(z_{y_0}),\varphi'(z_{y_0}))&\dots&0&0\\ \vdots&\ddots&\vdots&\vdots\\0&\dots&\beta(\varphi(z_{y_0}),\varphi'(z_{y_0}))&0\\0&\dots&0&\alpha(\varphi(z_{y_0}),\varphi'(z_{y_0}))\end{pmatrix}.
\end{equation*}
Therefore, first we have
\begin{equation*}
0\leq q(\varphi(z_{y_0}),\varphi'(z_{y_0}))+\varphi''(z_{y_0})\alpha(\varphi(z_{y_0}),\varphi'(z_{y_0}))
+\varphi'(z_{y_0})tr\left(\begin{pmatrix}0&C\\C&A_2\end{pmatrix}\begin{pmatrix}-X&0\\0&Y\end{pmatrix}\right),
\end{equation*}
where $C$ is an $n\times n$ matrix to be determined. Multiplying by $\frac{f^2(l)}{\varphi'(z_{y_0})\beta(\varphi(z_{y_0}),\varphi'(z_{y_0}))}$ gives
\begin{align*}
0&\leq \frac{f^2(l)}{\varphi'(z_{y_0})\beta(\varphi(z_{y_0}),\varphi'(z_{y_0}))}(q(\varphi(z_{y_0}),\varphi'(z_{y_0}))+\varphi''(z_{y_0})\alpha(\varphi(z_{y_0}),\varphi'(z_{y_0})))\\
&+\frac{f^2(l)}{\beta(\varphi(z_{y_0}),\varphi'(z_{y_0}))}tr\left(\begin{pmatrix}0&C\\C&A_2\end{pmatrix}\begin{pmatrix}-X&0\\0&Y\end{pmatrix}\right).
\end{align*}
Similarly, for the inequality at $x_0$ we get
\begin{align*}
0&\geq \frac{f^2(0)}{\varphi'(z_{x_0})\beta(\varphi(z_{x_0}),\varphi'(z_{x_0}))}(q(\varphi(z_{x_0}),\varphi'(z_{x_0}))+\varphi''(z_{x_0})\alpha(\varphi(z_{x_0}),\varphi'(z_{x_0})))\\
&-\frac{f^2(0)}{\beta(\varphi(z_{x_0}),\varphi'(z_{x_0}))}tr\left(\begin{pmatrix}A_1&0\\0&0\end{pmatrix}\begin{pmatrix}-X&0\\0&Y\end{pmatrix}\right).
\end{align*}
Combining them we obtain
\begin{align*}
0\leq& \frac{f^2(l)}{\varphi'\beta(\varphi,\varphi')}(q(\varphi,\varphi')+\varphi''\alpha(\varphi'))\bigg|_{z_{y_0}}
-\frac{f^2(0)}{\varphi'\beta(\varphi,\varphi')}(q(\varphi,\varphi')+\varphi''\alpha(\varphi,\varphi'))\bigg|_{z_{x_0}}\\
&+\frac{f^2(l)}{\beta(\varphi(z_{y_0}),\varphi'(z_{y_0}))}tr\left(\begin{pmatrix}0&C\\C&A_2\end{pmatrix}\begin{pmatrix}-X&0\\0&Y\end{pmatrix}\right)\\
&+\frac{f^2(0)}{\beta(\varphi(z_{x_0}),\varphi'(z_{x_0}))}tr\left(\begin{pmatrix}A_1&0\\0&0\end{pmatrix}\begin{pmatrix}-X&0\\0&Y\end{pmatrix}\right).
\end{align*}
Letting
\begin{equation*}
C=\begin{pmatrix}\frac{f(0)}{f(l)}\beta(\varphi(z_{y_0}),\varphi'(z_{y_0}))&&&\\&\ddots&&\\&&\frac{f(0)}{f(l)}\beta(\varphi(z_{y_0}),\varphi'(z_{y_0}))&\\&&&0\end{pmatrix},
\end{equation*}
then the matrix
\begin{align*}
W=&\quad\frac{f^2(l)}{\beta(\varphi(z_{y_0}),\varphi'(z_{y_0}))}\begin{pmatrix}0&C\\C&A_2\end{pmatrix}+\frac{f^2(0)}{\beta(\varphi(z_{x_0}),\varphi'(z_{x_0}))}\begin{pmatrix}A_1&0\\0&0\end{pmatrix}\\
=&\begin{pmatrix}f^2(0)I_{n-1}&0&f(0)f(l)I_{n-1}&0\\0&f^2(0)\frac{\alpha(\varphi,\varphi')}{\beta(\varphi,\varphi')}\big|_{z_{x_0}}&0&0\\
f(0)f(l)I_{n-1}&0&f^2(l)I_{n-1}&0\\0&0&0&f^2(l)\frac{\alpha(\varphi,\varphi')}{\beta(\varphi,\varphi')}\big|_{z_{y_0}}\end{pmatrix}
\end{align*}
is positive semidefinite.

So we can use
\begin{equation*}
\begin{pmatrix}-X&0\\0&Y\end{pmatrix}\leq M+\lambda M^2
\end{equation*}
to get
\begin{align*}
0&\leq \frac{f^2(l)}{\varphi'\beta(\varphi,\varphi')}(q(\varphi,\varphi')+\varphi''\alpha(\varphi,\varphi'))\bigg|_{z_{y_0}}
-\frac{f^2(0)}{\varphi'\beta(\varphi,\varphi')}(q(\varphi,\varphi')+\varphi''\alpha(\varphi,\varphi'))\bigg|_{z_{x_0}}\\
&+tr(WM)+\lambda\ tr(WM^2).
\end{align*}
Now we compute $tr(WM)$ as follows.
\begin{align*}
tr(WM)&=\sum_{i=1}^{n-1}D^2\tilde{d}((f(0)e_i(0),f(l)e_i(l)),(f(0)e_i(0),f(l)e_i(l)))\\
      &+\frac{\alpha(\varphi,\varphi')}{\beta(\varphi,\varphi')}\bigg|_{z_{x_0}}D^2\tilde{d}((f(0)e_n(0),0),(f(0)e_n(0),0))\\
      &+\frac{\alpha(\varphi,\varphi')}{\beta(\varphi,\varphi')}\bigg|_{z_{y_0}}D^2\tilde{d}((0,f(l)e_n(l)),(0,f(l)e_n(l))).
\end{align*}
Note that
\begin{align*}
\quad &D^2\tilde{d}((f(0)e_i(0),f(l)e_i(l)),(f(0)e_i(0),f(l)e_i(l)))\\
=&\frac{d^2}{dt^2}\bigg|_{t=0}\tilde{d}(\exp_{x_0}(tf(0)e_i(0)),\exp_{y_0}(tf(l)e_i(l)))\\
=&\frac{d^2}{dt^2}\bigg|_{t=0}L(\exp_{\gamma_0(s)}(tf(s)e_i(s))_{s\in [0,l]})\\
=&\int_0^l[(f'(s))^2-f^2(s)R(e_n,e_i,e_i,e_n)]ds\\
=&ff'\big|_0^l-\int_0^lf(f''+fR(e_n,e_i,e_i,e_n))ds,
\end{align*}
which implies
\begin{align*}
\sum_{i=1}^{n-1} &D^2\tilde{d}((f(0)e_i(0),f(l)e_i(l)),(f(0)e_i(0),f(l)e_i(l)))\\
=&(n-1)ff'\big|_0^l-\int_0^lf((n-1)f''+fRic(e_n,e_n))ds\\
\leq& (n-1)ff'\big|_0^l-\int_0^l(n-1)f(f''+\kappa f)ds.
\end{align*}
Similarly we get
\begin{align*}
&D^2\tilde{d}((f(0)e_n(0),0),(f(0)e_n(0),0))=0,\\
&D^2\tilde{d}((0,f(l)e_n(l)),(0,f(l)e_n(l)))=0.
\end{align*}
In summary, we have
\begin{align}
0&\leq \frac{f^2(l)}{\varphi'\beta(\varphi,\varphi')}(q(\varphi,\varphi')+\varphi''\alpha(\varphi,\varphi'))\bigg|_{z_{y_0}}
-\frac{f^2(0)}{\varphi'\beta(\varphi,\varphi')}(q(\varphi,\varphi')+\varphi''\alpha(\varphi,\varphi'))\bigg|_{z_{x_0}}\nonumber\\
&+(n-1)ff'\big|_0^l-\int_0^l(n-1)f(f''+\kappa f)ds+\lambda\ tr(WM^2).\label{comparison}
\end{align}
Taking $f(s)\equiv 1$ and $\kappa=0$, and letting $\lambda\rightarrow 0$, we have
\begin{equation*}
0\leq \frac{q(\varphi,\varphi')+\varphi''\alpha(\varphi,\varphi')}{\varphi'\beta(\varphi,\varphi')}\bigg|^{z_{y_0}}_{z_{x_0}}.
\end{equation*}
Now taking Condition \eqref{cond3.2} into account, since $z_{y_0}=z_{x_0}+d(x_0,y_0)+\e_0>z_{x_0}$, we get a contradiction. Then we must have
\begin{equation*}
Z(x,y)=\psi(u(y))-\psi(u(x))-d(x,y)\leq 0,
\end{equation*}
which is the desired result.

\end{proof}

\section{Manifolds with negative lower Ricci curvature bound}\label{sec4}

In this section we prove the following modulus of continuity estimate.
\begin{thm}\label{thm4.1}
Let $(M^n,g)$ be a compact Riemannian manifold with Ricci curvature $Ric\geq (n-1)\kappa$, $\kappa< 0$, and $u$ a viscosity solution of Equation \eqref{eq-G}. Suppose the barrier $\varphi:[a,b]\rightarrow [\inf u,\sup u]$ satisfies
\begin{align}
&\varphi'>0,\label{cond4.1}\\
\frac{q(\varphi,\varphi')+\varphi''\alpha(\varphi,\varphi')}{\varphi'\beta(\varphi,\varphi')}\bigg|_z&+(n-1)\frac{\rho'}{\rho}=0,\label{cond4.2}
\end{align}
where $\rho:[a,b]\rightarrow \R^+$ satisfies $\rho''+\kappa \rho=0$ and $(\frac{\rho'}{\rho})'>0$. Moreover let $\psi$ be the inverse of $\varphi$, i.e. $\psi(\varphi(z))=z$. Then we have
\begin{equation*}
\psi(u(y))-\psi(u(x))-d(x,y)\leq 0,\: \forall x,y\in M.
\end{equation*}
\end{thm}
By letting $y$ approach $x$, we get the following gradient bound.
\begin{corr}\label{corr4.1}
Under the conditions of Theorem \ref{thm4.1}, if moreover $u\in C^1(M)$, then for every $x\in M$ we have
\begin{equation*}
|\n u(x)|\leq \varphi'(\psi(u(x))).
\end{equation*}
\end{corr}

\begin{rem}
In fact, Theorem \ref{thm4.1} is an equivalent statement of Theorem \ref{thm2}. First, in a warped product $\bar{M}=N\times [a,b]$ with $\bar{g}=ds^2+\rho(s)^2g^N$, the requirement $Ric(\partial_s,\partial_s)=(n-1)\kappa$ is equivalent to $\rho''+\kappa \rho=0$. Second, by the calculation in \cite[Sec. 3]{And14}, $\bar{u}(x,s)=\varphi(s)$ being a solution of \eqref{eq-G} means
\begin{align*}
\alpha(\varphi,\varphi')\varphi''+q(\varphi,\varphi')&+(n-1)\varphi'\beta(\varphi,\varphi')\frac{\rho'}{\rho}=0.
\end{align*}
\end{rem}
\begin{rem}\label{rem1}
Note that $(\frac{\rho'}{\rho})'>0$ with $\rho''+\kappa \rho=0$ implies $\kappa<-(\rho')^2/\rho^2$ and so necessarily $\kappa<0$. In addition, it is easy to see that $\rho(z)=\cosh (z_0+z)$ satisfies our conditions; while $\rho(z)=\cos z$ does not, which means for $\kappa>0$ (or for other choices of $\rho$ with $\kappa<0$) we need a different argument.
\end{rem}

\begin{proof}[Proof of Theorem \ref{thm4.1}]
The proof is by contradiction. Define
\begin{equation*}
Z(x,y)=\psi(u(y))-\psi(u(x))-d(x,y).
\end{equation*}
Assume otherwise
\begin{equation*}
\max_{x,y\in M} Z(x,y)=Z(x_0,y_0)=\e_0>0.
\end{equation*}
As derived in the last section, we have
\begin{align*}
0&\leq \frac{f^2(l)}{\varphi'\beta(\varphi,\varphi')}(q(\varphi,\varphi')+\varphi''\alpha(\varphi,\varphi'))\bigg|_{z_{y_0}}
-\frac{f^2(0)}{\varphi'\beta(\varphi,\varphi')}(q(\varphi,\varphi')+\varphi''\alpha(\varphi,\varphi'))\bigg|_{z_{x_0}}\\
&+(n-1)ff'\big|_0^l-\int_0^l(n-1)f(f''+\kappa f)ds.
\end{align*}

Take $f(s)=\rho(z_{x_0}+\e_0+s)$. And \eqref{cond4.2} shows that
\begin{equation*}
\frac{q(\varphi,\varphi')+\varphi''\alpha(\varphi,\varphi')}{\varphi'\beta(\varphi,\varphi')}\bigg|_z=-(n-1)\frac{\rho'(z)}{\rho(z)}
\end{equation*}
is strictly decreasing. So we have
\begin{align*}
0&< \frac{f^2(l)}{\varphi'\beta(\varphi,\varphi')}(q(\varphi,\varphi')+\varphi''\alpha(\varphi,\varphi'))\bigg|_{z_{y_0}}
-\frac{f^2(0)}{\varphi'\beta(\varphi,\varphi')}(q(\varphi,\varphi')+\varphi''\alpha(\varphi,\varphi'))\bigg|_{z_{x_0}+\e_0}\\
 &+(n-1)ff'\bigg|_0^l-\int_0^l (n-1)f(f''+\kappa f)ds\\
 &=\rho^2\left(\frac{q(\varphi,\varphi')+\varphi''\alpha(\varphi,\varphi')}{\varphi'\beta(\varphi,\varphi')}+(n-1)\frac{\rho'}{\rho}\right)\bigg|_{z_{x_0}+\e_0}^{z_{y_0}}\\
 &=0,
\end{align*}
which is a contradiction.
Then we must have
\begin{equation*}
Z(x,y)=\psi(u(y))-\psi(u(x))-d(x,y)\leq 0,
\end{equation*}
which is the desired result.
\end{proof}

\section{Riemannian manifolds with general lower Ricci bound}\label{sec5}

\subsection{The result and an example}\label{sec:method2}

As observed in Remark \ref{rem1}, the case $\kappa>0$ may not be handled as in the last two sections. In other words, for this case we may not prove that any solution $\varphi$ to the one-dimensional equation is a barrier. However, we find that for some family of the one-dimensional solutions, the property of being barriers can be extended smoothly in the family, and in fact this phenomenon holds for any $\kappa\in \R$ regardless of its sign. More precisely, we prove the following result.
\begin{thm}\label{thm5.1}
Let $(M^n,g)$ be a compact Riemannian manifold with Ricci curvature $Ric\geq (n-1)\kappa$, $\kappa\in \R$. Assume in Equation \eqref{eq-G} the coefficients $\alpha$, $\beta$ are $C^2$ functions, and $q$ a continuous function. Let $u$ be a $C^3$ solution of Equation \eqref{eq-G}. Suppose there exists a family of functions $\varphi_c$ from $[a_c,b_c]$ to $[\inf u,\sup u]$ which satisfies
\begin{align*}
\frac{\alpha(\varphi,\varphi')\varphi''+q(\varphi,\varphi')}{\varphi'\beta(\varphi,\varphi')}&+(n-1)\frac{\rho'}{\rho}=0,\\
\varphi'&>0,
\end{align*}
and depends smoothly on $c\in (c_u,+\infty)$. Here $\rho$ is a positive $C^2$ function on $[a_c,b_c]$ satisfying $\rho''+\kappa \rho=0$. Moreover, assume for $c\gg c_u$, $\varphi_c'$ is uniformly large. Then for any $c>c_u$ we have
\begin{equation}\label{eq5.1}
\psi_c(u(y))-\psi_c(u(x))-d(x,y)\leq 0,\: \forall x,y\in M,
\end{equation}
where $\psi_c$ is the inverse of $\varphi_c$.
\end{thm}
By letting $y$ approach $x$, we get the following gradient bound.
\begin{corr}\label{corr5.1}
Under the conditions of Theorem \ref{thm5.1}, for every $x\in M$ we have
\begin{equation*}
|\n u(x)|\leq \varphi_c'(\psi_c(u(x))),\: c>c_u.
\end{equation*}
\end{corr}

Here we give an example in which Theorem \ref{thm5.1} applies.
\begin{exam}
Let us consider the following equation:
\begin{equation*}
div_g(\Phi'(|\nabla_g u|^2) \n_g u)+q(u)=0,
\end{equation*}
where $q(u)=Q'(u)$ for some function $Q$, on a Riemannian manifold with $Ric\geq n-1$. Here $\Phi$ satisfies some structure conditions in Subsection \ref{sec6.5}. We also use the following notations:
\begin{align*}
K(s)&:=\Phi'(s)s-\f{1}{2}\Phi(s),\\
\Lambda(s)&:=2\Phi''(s)s+\Phi'(s).
\end{align*}
Note that for this example, $\alpha(u,|\nabla u|)=\Lambda(|\nabla u|^2)$ and $\beta(u,|\nabla u|)=\Phi'(|\nabla u|^2)$.

Let $c_u=sup_{r\in [\inf u,\sup u]}Q(r)$ and suppose $u_0\in [\inf u,\sup u]$ is such that $Q(u_0)=c_u$. Then the ODE appearing in Theorem \ref{thm5.1} becomes the following: (Here $c>c_u$ and $\rho (z)=\cos z$.)
\begin{equation}\label{eq-b1}
\begin{cases}
\Lambda((\varphi_c')^2)\varphi_c''+q(\varphi_c)=(n-1)\tan z\cdot \Phi'((\varphi_c')^2)\varphi_c',\\
\varphi(0)=u_0,\:\varphi'(0)= \sqrt{K^{-1}\circ (c-c_u)},\\
\end{cases}
\end{equation}
which has a unique solution $\varphi_c:[a,b]\rightarrow [\inf u,\sup u]$, where $0\in [a,b]\subset (-\pi/2,\pi/2)$. Here $K^{-1}$ denotes the inverse function of $K$. Equivalently, $\varphi_c$ satisfies:
\begin{equation*}
K((\varphi')^2)+Q(\varphi)=c+\int_0^z (n-1)\tan z\cdot \Phi'((\varphi')^2)(\varphi')^2dz.
\end{equation*}
Since $\int_0^z (n-1)\tan z\cdot \Phi'((\varphi')^2)(\varphi')^2dz\geq 0$, we always have
\begin{equation*}
\varphi_c'(z)\geq \sqrt{K^{-1}\circ(c-Q(\varphi_c))}\geq \sqrt{K^{-1}\circ(c-c_u)}>0.
\end{equation*}
Moreover, when $c\gg c_u$, $\varphi_c'$ is uniformly large. So by Theorem \ref{thm5.1}, for the corresponding inverse $\psi_c$ ($c>c_u$), we have
\begin{equation*}
\psi_{c}(u(y))-\psi_{c}(u(x))-d(x,y)\leq 0,\: \forall x,y\in M.
\end{equation*}
\end{exam}

\subsection{The proof} In this subsection, we first give the outline of the proof and then derive a computational lemma which is needed in the proof.


\begin{proof}[Proof of Theorem \ref{thm5.1}]

Assume by contradiction that \eqref{eq5.1} does not hold for some $c_0>c_u$. Then for any $c>c_0$ we can solve by perturbation
\begin{equation}\label{eq-b2}
\begin{cases}
\dfrac{\alpha(\varphi,\varphi')\varphi''+q(\varphi,\varphi')}{\varphi'\beta(\varphi,\varphi')}=-(n-1)\dfrac{\rho'}{\rho}-\delta(c)\dfrac{z}{\rho^2},\\
\varphi(a_c)=\varphi_c(a_c),\:\varphi'(a_c)=\varphi_c'(a_c),\\
\end{cases}
\end{equation}
to get a function $\varphi_{c,\delta(c)}$, where $\delta(c)>0$ is small, and depends on $c$ with $\lim_{c\rightarrow c_0} \delta(c)=0$.

Denote $D=\{(x,x):x\in M\}$. Then we consider a manifold $\hat{M}$ with boundary which compactifies $(M\times M)\setminus D$ as follows: As a set, $\hat{M}$ is the disjoint union of $(M\times M)\setminus D$ with the unit sphere bundle $SM=\{(x,v)\in TM:||v||=1\}$. The manifold-with-boundary structure is defined by the atlas generated by all charts for $(M\times M)\setminus D$, together with the charts $\hat{Y}$ from $SM\times (0,r)$ defined by taking a chart $Y$ for $SM$, and setting $\hat{Y}(z,s):=(\exp(sY(z)),\exp(-sY(z)))$.

In the following we write $\varphi=\varphi_{c,\delta(c)}$ and $\psi=\psi_{c,\delta(c)}$ for short. Let
\begin{equation*}
Z(x,y)=\psi(u(y))-\psi(u(x))-d(x,y).
\end{equation*}
We define a function $\hat{Z}$ on $\hat{M}$ as follows: For $(x,y)\in (M\times M)\setminus D$, we define
\begin{equation*}
\hat{Z}(x,y)=\f{Z(x,y)}{d(x,y)}.
\end{equation*}
For $(x,v)\in SM$, we define
\begin{equation*}
\hat{Z}(x,v)=\f{1}{\varphi'(z_x)}D_vu(x)-1.
\end{equation*}

Then one can check that $\hat{Z}$ is a continuous function on $\hat{M}$. And when $c\gg c_0$, we have $\hat{Z}\leq 0$ on $\hat{M}$.

Thus we can define
\begin{equation*}
c_1:=\inf \{\bar{c}>c_0 \big| \hat{Z}\leq 0 \text{ on }\hat{M} \text{ for }c\in (\bar{c},+\infty)\}.
\end{equation*}
By assumption we have $c_1>c_0$, which we shall prove leads to a contradiction.

In fact for $c=c_1$ there will be two cases.

\textbf{Case 1}: $0=\hat{Z}(x_0,y_0)$ for some $x_0\neq y_0$, i.e. $Z(x_0,y_0)=0$.

\textbf{Case 2}: $Z(x,y)<0$ for any $x\neq y\in M$ and $\hat{Z}(x_0,v_0)=0$ for some $(x_0,v_0)\in SM$.

We will rule out these two cases, so $c_1>c_0$ is impossible.

For that purpose, we need the following lemma:
\begin{lem}\label{lem-computation}
Let $u$ be a $C^3$ solution of Equation \eqref{eq-G}.  Let $x\neq y$ with $d(x,y)<\textrm{inj}(M)$, the injectivity radius of $M$.  Let $\gamma_0:[0,l]\rightarrow M$ be the length-minimizing geodesic from $x$ to $y$, and choose Fermi coordinates along $\gamma_0$ as before.   Fix a $C^2$ function $f:[0,l]\rightarrow \R$. Then $Z=Z(x,y)$ satisfies the following equation
\begin{align*}
\mathcal{L}Z:&=f(l)^2\f{\alpha(\varphi,\varphi')}{\beta(\varphi,\varphi')}\bigg|_{z_y}
D^2_{(0,e_n),(0,e_n)}Z
+f(0)^2\f{\alpha(\varphi,\varphi')}{\beta(\varphi,\varphi')}\bigg|_{z_x}
D^2_{(e_n,0),(e_n,0)}Z\\
&\quad\null+\sum_{i<n}
D^2_{(f(0)e_i,f(l)e_i),(f(0)e_i,f(l)e_i)}Z\\
&=-f(l)^2\f{\alpha(\varphi,\varphi')\varphi''+q(\varphi,\varphi')}{\varphi'\beta(\varphi,\varphi')}\bigg|_{z_y}+f(0)^2\f{\alpha(\varphi,\varphi')\varphi''+q(\varphi,\varphi')}{\varphi'\beta(\varphi,\varphi')}\bigg|_{z_x}\\
&-(n-1)f(s)f'(s)\bigg|_0^l+\int_0^l\left((n-1)ff''+f^2Ric(e_n)\right)ds+DZ*DZ+P\cdot DZ,
\end{align*}
where the coefficients of $DZ*DZ$ and $DZ$ are $C^1$ functions.
\end{lem}
With this lemma at hand, choose
\begin{equation*}
f(s)=\rho (z_x+Z+s).
\end{equation*}
Using the condition on Ricci curvature, we can derive
\begin{align}\label{eq5.4}
&\mathcal{L}Z+DZ*DZ+P\cdot DZ\nonumber\\
&\geq -f(l)^2\f{\alpha(\varphi,\varphi')\varphi''+q(\varphi,\varphi')}{\varphi'\beta(\varphi,\varphi')}\bigg|_{z_y}+f(0)^2\f{\alpha(\varphi,\varphi')\varphi''+q(\varphi,\varphi')}{\varphi'\beta(\varphi,\varphi')}\bigg|_{z_x}-(n-1)f(s)f'(s)\bigg|_0^l\nonumber\\
&=-\rho^2\left(\f{\alpha(\varphi,\varphi')\varphi''+q(\varphi,\varphi')}{\varphi'\beta(\varphi,\varphi')}+(n-1)\frac{\rho'}{\rho}\right)\bigg|_{z_x+Z}^{z_y}\nonumber\\
&\quad+f(0)^2\f{\alpha(\varphi,\varphi')\varphi''+q(\varphi,\varphi')}{\varphi'\beta(\varphi,\varphi')}\bigg|^{z_x}_{z_x+Z}.
\end{align}

Let us first consider Case 1. The same computation applies and so by the maximum principle we have at $(x_0,y_0)$
\begin{align*}
0\geq &\mathcal{L}Z+DZ*DZ+P\cdot DZ\\
&=-\rho^2 \left(\f{\alpha(\varphi,\varphi')\varphi''+q(\varphi,\varphi')}{\varphi'\beta(\varphi,\varphi')}+(n-1)\frac{\rho'}{\rho}\right)\bigg|_{z_x}^{z_y}\\
&=\delta(c_1)z\big|_{z_x}^{z_y}=\delta(c_1)l>0.
\end{align*}
This contradiction shows that Case 1 cannot occur.

Next we consider Case 2: $Z(x,y)<0$ for any $x\neq y\in M$. In this case by \eqref{eq5.4} and the choice of $\varphi_{c,\delta(c)}$, we have, for $x\neq y$ close enough to each other,
\begin{align*}
&\mathcal{L}Z+DZ*DZ+P\cdot DZ\\
&\geq \delta(c_1)z\big|_{z_x+Z}^{z_y}+f(0)^2\left(-(n-1)\frac{\rho'}{\rho}-\delta(c)\frac{z}{\rho^2}\right)\bigg|^{z_x}_{z_x+Z}\\
&=\delta(c_1)l+ C(x,y)Z\geq C(x,y)Z,
\end{align*}
where $C(x,y)$ is some bounded function. That is, the inequality above holds near the boundary of $\hat{M}$. Now recall the boundary Hopf maximum principle from \cite{Hil70}, applying to any $C^2$ function $Z$ which has a strict maximum boundary value zero and satisfies $a^{ij}Z_{ij}+b^iZ_i+cZ\geq 0$ with $a^{ij}\in C^2$, $b^i\in C^1$, $c\in L^\infty$ and $[a^{ij}]\geq 0$. Thus we derive for $(x_0,v_0)$
\begin{align*}
0>D_{(0,v)}Z(x,x)&=\lim_{t\rightarrow 0}\f{Z(x,\exp_x(tv))-Z(x,x)}{t}\\
               &=\lim_{t\rightarrow 0}\f{\psi(u(\exp_x(tv)))-\psi(u(x))-d(x,\exp_x(tv))}{t}\\
               &=\f{1}{\varphi'(z_x)}D_vu(x)-1,
\end{align*}
which contradicts $\hat{Z}(x_0,v_0)=\f{1}{\varphi'(z_{x_0})}D_{v_0}u(x_0)-1=0$. So Case 2 is also ruled out. So $c_1>c_0$ is impossible, and the proof of Theorem~\ref{thm5.1} is complete.
\end{proof}
\hspace{10mm}

Finally we give the proof of Lemma \ref{lem-computation}.
\begin{proof}[The proof of Lemma \ref{lem-computation}]

Recall
\begin{equation*}
Z(x,y)=\psi(u(y))-\psi(u(x))-d(x,y).
\end{equation*}

For any $X\in T_xM$ and $Y\in T_y M$, there exists a variation $\gamma(\e,s)$ of $\gamma_0(s)$ such that $\gamma_\e(0)=X$ and $\gamma_\e(l)=Y$. Then the first derivative of $Z$ in the direction $(X,Y)$ is
\begin{align*}
D_{(X,Y)}Z&=\psi'(u(y))\langle \n u(y),\gamma_\e(l)\rangle-\psi'(u(x))\langle \n u(x),\gamma_\e(0)\rangle-\langle T(s),\gamma_\e(s)\rangle|_0^l\\
          &=\langle \psi'(u)\n u-\gamma_s,\gamma_\e(s)\rangle|_0^l.
\end{align*}

Furthermore the second derivative of $Z$ is
\begin{align*}
D^2_{(X,Y),(X,Y)}Z
&= \psi''(u)\langle \n u,\gamma_\e(l)\rangle^2+\psi'(u)\left(\langle D_{\gamma_\e}(\n u),\gamma_\e(l)\rangle+\langle \n u,D_{\gamma_\e}\gamma_\e(l)\rangle\right)\\
&-\psi''(u)\langle \n u,\gamma_\e(0)\rangle^2-\psi'(u)\left(\langle D_{\gamma_\e}(\n u),\gamma_\e(0)\rangle+\langle \n u,D_{\gamma_\e}\gamma_\e(0)\rangle\right)\\
&-\int_0^l\left( |\n_{\gamma_s} (\gamma_\e^\perp)|^2-R(\gamma_s,\gamma_\e,\gamma_\e,\gamma_s)\right)ds-\langle \gamma_s,\n_{\gamma_\e} \gamma_\e \rangle\big|_0^l.
\end{align*}

Note that
\begin{align*}
&\quad \psi'\cdot \varphi'=1,\\
\psi''\cdot &(\varphi')^2+\psi'\cdot \varphi''=0.
\end{align*}
We get
\begin{align*}
D^2_{(X,Y),(X,Y)}Z&= -\f{\varphi''}{(\varphi')^3}\langle \n u,\gamma_\e(l)\rangle^2+\f{\varphi''}{(\varphi')^3}\langle \n u,\gamma_\e(0)\rangle^2\\
&+\f{1}{\varphi'}D^2u(\gamma_\e(l),\gamma_\e(l))-\f{1}{\varphi'}D^2u(\gamma_\e(0),\gamma_\e(0))\\
&-\int_0^l\left( |\n_{\gamma_s} (\gamma_\e^\perp)|^2-R(\gamma_s,\gamma_\e,\gamma_\e,\gamma_s)\right)ds+D_{(D_{\gamma_\e}\gamma_\e(0),D_{\gamma_\e}\gamma_\e(l))}Z.
\end{align*}
Now we choose particular variations to obtain inequalities for particular parts of the Hessian of $Z$:

\noindent{\bf(1).} Vary $y$: Choose the variation
\begin{equation*}
\gamma(\e,s)=\gamma_0(s+\e\f{s}{l}).
\end{equation*}
So $\gamma_\e(l)=e_n$, $\gamma_\e(0)=0$. Then we get
\begin{align*}
D_{(0,e_n)}Z&=\langle \f{1}{\varphi'}\n u-\gamma_s,e_n\rangle (l)=\f{u_n(y)}{\varphi'}-1,\\
D^2_{(0,e_n),(0,e_n)}Z&=-\f{\varphi''u_n^2(y)}{(\varphi')^3}+\f{1}{\varphi'}u_{nn}(y)\notag\\
&=-\f{\varphi''}{\varphi'}(1+D_{(0,e_n)}Z)^2+\f{1}{\varphi'}u_{nn}(y).
\end{align*}

\noindent{\bf(2).} Vary $x$: Choose the variation
\begin{equation*}
\gamma(\e,s)=\gamma_0(s+\e\f{l-s}{l}).
\end{equation*}
So $\gamma_\e(l)=0$, $\gamma_\e(0)=e_n$. Then similarly we get
\begin{align*}
D_{(e_n,0)}Z&=-\langle \f{1}{\varphi'}\n u-\gamma_s,e_n\rangle (0)=-\f{u_n(x)}{\varphi'}+1,\\
D^2_{(e_n,0),(e_n,0)}Z&=\f{\varphi''u_n^2(x)}{(\varphi')^3}-\f{1}{\varphi'}u_{nn}(x)=\f{\varphi''}{\varphi'}(1-D_{(e_n,0)}Z)^2-\f{1}{\varphi'}u_{nn}(x).
\end{align*}

\noindent{\bf(3).} Vary $\gamma_0$ along $e_i(s)$ for fixed $i<n$: Choose
\begin{equation*}
\gamma(\e,s)=\exp_{\gamma_0(s)}(\e f(s) e_i(s)).
\end{equation*}
So $\gamma_\e(s)=f(s)e_i(s)$. Therefore
\begin{align*}
D_{(f(0)e_i,f(l)e_i)}Z&=\langle \f{1}{\varphi'}\n u-\gamma_s,f(s)e_i(s)\rangle|_0^l\\
                      &=\f{f(l)}{\varphi'}u_i(y)-\f{f(0)}{\varphi'}u_i(x),\\
D^2_{(f(0)e_i,f(l)e_i),(f(0)e_i,f(l)e_i)}Z&=-\f{\varphi''}{(\varphi')^3}f^2(l)u_i^2(y)+\f{\varphi''}{(\varphi')^3}f^2(0)u_i^2(x)\\
&+\f{1}{\varphi'}f^2(l)u_{ii}(y)-\f{1}{\varphi'}f^2(0)u_{ii}(x)\\
&-\int_0^l[(f'(s))^2-f^2(s)R(e_n,e_i,e_i,e_n)]ds.
\end{align*}
Then after summation from $i=1$ to $i=n-1$ we have
\begin{align*}
\sum_{i<n}D^2_{(f(0)e_i,f(l)e_i),(f(0)e_i,f(l)e_i)}Z&=-\f{\varphi''}{(\varphi')^3}f^2(l)\sum_{i<n}u_i^2(y)+\f{\varphi''}{(\varphi')^3}f^2(0)\sum_{i<n}u_i^2(x)\\
&+\f{1}{\varphi'}f^2(l)\sum_{i<n}u_{ii}(y)-\f{1}{\varphi'}f^2(0)\sum_{i<n}u_{ii}(x)\\
&-(n-1)f(s)f'(s)|_0^l+\int_0^l[(n-1)ff''+f^2Ric(e_n)]ds.
\end{align*}
Now recall that at $x_0$ or $y_0$ Equation \eqref{eq-G} is
\begin{equation*}
\alpha(\varphi,\varphi') u_{nn}+\beta(\varphi,\varphi') \sum_{i<n} u_{ii}+q(\varphi)+DZ*DZ+P\cdot DZ=0.
\end{equation*}

Therefore direct computation yields
\begin{align*}
\mathcal{L}Z:
=&-f(l)^2\f{\alpha(\varphi,\varphi')\varphi''}{\varphi'\beta(\varphi,\varphi')}\bigg|_{z_y}(1+D_{(0,e_n)}Z)^2+f(0)^2\f{\alpha(\varphi,\varphi')\varphi''}{\varphi'\beta(\varphi,\varphi')}\bigg|_{z_x}(1-D_{(e_n,0)}Z)^2\\
&-\f{\varphi''}{\varphi'}\bigg|_{z_y}\sum_{i<n}(D_{(0,f(l)e_i)}Z)^2+\f{\varphi''}{\varphi'}\bigg|_{z_x}\sum_{i<n}(D_{(f(0)e_i,0)}Z)^2\\
&-f(l)^2 \f{q(\varphi)}{\varphi'\beta(\varphi,\varphi')}\bigg|_{z_y}+f(0)^2 \f{q(\varphi)}{\varphi'\beta(\varphi,\varphi')}\bigg|_{z_x}-(n-1)f(s)f'(s)\bigg|_0^l\\
&+\int_0^l\left((n-1)ff''+f^2Ric(e_n)\right)ds+DZ*DZ+P\cdot DZ.
\end{align*}
Putting the terms involving $DZ$ together, we complete the proof of Lemma~\ref{lem-computation}.

\end{proof}

\section{Finsler spaces with nonnegative weighted Ricci curvature}\label{sec6}

In this section we consider the problem on Finsler manifolds, although only for Equation \eqref{eq6.0} below of divergence form. We first briefly review the fundamentals of Finsler geometry \cite{BCS00, She01} and some developments from \cite{Oht09,OS14}. Then we give the structure conditions for the equation and regularity of its solutions. Finally we will discuss the modulus of continuity estimates in this Finsler context.

\subsection{Finsler manifolds}

Let $M$ be an $n$-dimensional connected smooth manifold without boundary. Given a local coordinate $\{x^i\}_{i=1}^n$ on an open set $U\subset M$, let $\{x^i, V^j\}_{i,j=1}^n$ be the coordinate of $TU$, i.e.
\begin{equation*}
V=V^j\f{\pt}{\pt x^j}, \forall V\in T_x M, x\in U.
\end{equation*}
\begin{defn}[Finsler structures]
A function $F:TM\rightarrow [0,\infty)$ is called a Finsler structure if the following three conditions hold:
\begin{enumerate}
  \item (Regularity) $F$ is $C^\infty$ on $TM\setminus 0$;
  \item (Positive $1$-homogeneity) $F(x,cV)=cF(x,V)$ for all $(x,V)\in TM$ and all $c> 0$;
  \item (Strong convexity) The matrix
  \begin{equation*}
  g_{ij}(x,V):=\f{\pt^2}{\pt V^i \pt V^j}\left(\f{1}{2}F^2\right)(x,V)
  \end{equation*}
  is positive definite for all $(x,V)\in TM\setminus 0$.
\end{enumerate}
\end{defn}
We call such a pair $(M,F)$ a smooth Finsler manifold. If moreover a measure $m$ is given on $M$, we call the triple $(M,F,m)$ a Finsler measure space. Note that for every non-vanishing vector field $V$, $g_{ij}(x,V)$ induces a Riemannian structure $g_V$ on $M$ by the following formula
\begin{equation*}
g_V(X,Y)=g_{ij}(x,V)X^iY^j, \forall X,Y\in T_x M.
\end{equation*}
By the homogeneity of $F$, $g_V(V,V)=F^2(x,V)$.

For $x,y\in M$, the distance function from $x$ to $y$ is defined by
\begin{equation*}
d(x,y):=\inf_\gamma \int_0^1 F(\gamma(t),\dot{\gamma}(t))dt,
\end{equation*}
where the infimum is taken over all $C^1$-curves $\gamma:[0,1]\rightarrow M$ such that $\gamma(0)=x$ and $\gamma(1)=y$. Note that generally $d(x,y)\neq d(y,x)$ since $F$ is only positively homogeneous.

A $C^\infty$-curve $\gamma:[0,1]\rightarrow M$ is called a geodesic if it is locally minimizing and has a constant speed (i.e. $F(\gamma(t),\dot{\gamma}(t))$ is constant). For $V\in T_xM$, if there exists a geodesic $\gamma:[0,1]\rightarrow M$ with $\dot{\gamma}(0)=V$, then we define the exponential map by $\exp_x(V):=\gamma(1)$. We say that $(M,F)$ is forward complete if the exponential map is defined on whole $TM$. Then by Hopf-Rinow theorem (see \cite{BCS00}), any pair of points can be connected by a minimal geodesic.

\subsection{Chern connection and Ricci curvature}

Let $\pi:TM\rightarrow M$ be the projection. There exists a unique linear connection on $\pi^*TM$, which is called Chern connection. The Chern connection is determined by the following structure equations:
\begin{enumerate}
  \item Torsion freeness:
  \begin{equation*}
  D_X^VY-D_Y^VX=[X,Y];
  \end{equation*}
  \item Almost $g$-compatibility:
  \begin{equation*}
  Z(g_V(X,Y))=g_V(D_Z^VX,Y)+g_V(X,D_Z^VY)+2C_V(D_Z^VV,X,Y),
  \end{equation*}
  for $V\in TM\setminus 0,X,Y,Z\in TM$.
\end{enumerate}
Here $D^V_XY$ is the covariant derivative with respect to the reference vector $V$ and
\begin{equation*}
C_V(X,Y,Z):=C_{ijk}(V)X^iY^jZ^k=\f{1}{4}\f{\pt^3 F^2(x,V)}{\pt V^i\pt V^j\pt V^k}X^iY^jZ^k
\end{equation*}
is the Cartan tensor of $(M,F)$. Note that $D^{cV}_XY=D^V_XY$, $c>0$ (see e.g. (2.5) in \cite{OS14}), and $C_V(V,X,Y)=0$ due to the homogeneity of $F$.

Given two linear independent vectors $V,W\in T_xM\setminus 0$, the flag curvature is defined by
\begin{equation*}
K^V(V,W)=\f{g_V(R^V(V,W)W,V)}{g_V(V,V)g_V(W,W)-g_V(V,W)^2},
\end{equation*}
where $R^V$ is the Riemannian curvature given by
\begin{equation*}
R^V(X,Y)Z:=D_X^VD_Y^VZ-D_Y^VD_X^VZ-D_{[X,Y]}^VZ.
\end{equation*}
Then the Ricci curvature is defined by
\begin{equation*}
Ric(V):=\sum_{i=1}^{n-1}K^V(V,e_i),
\end{equation*}
where $\{e_1,\cdots,e_{n-1},\f{V}{F(V)}\}$ is the orthonormal basis of $T_xM$ with respect to $g_V$. (Note that $Ric$ is $0$-homogeneous.)

Next we recall the definition of the weighted Ricci curvature on Finsler manifolds introduced by Ohta in \cite{Oht09}.
\begin{defn}
Given a unit vector $V\in T_xM$, let $\gamma:(-\e,\e)\rightarrow M$ be the geodesic with $\gamma(0)=x$ and $\dot{\gamma}(0)=V$. We decompose the measure $dm$ as $dm=e^{-\Psi}vol_{\dot{\gamma}}$ along $\gamma$, where $vol_{\dot{\gamma}}$ is the volume form of $g_{\dot{\gamma}}$. Define the weighted Ricci curvature as
\begin{enumerate}
  \item $Ric_n(V):=\begin{cases}Ric(V)+(\Psi\circ \gamma)''(0),\text{ if } (\Psi\circ \gamma)'(0)=0,\\-\infty,\text{ otherwise, }\end{cases}$
  \item $Ric_N(V):=Ric(V)+(\Psi\circ \gamma)''(0)-\dfrac{(\Psi\circ \gamma)'(0)^2}{N-n}\text{ for }N\in (n,\infty),$
  \item $Ric_\infty(V):=Ric(V)+(\Psi\circ \gamma)''(0).$
\end{enumerate}
For $c\geq 0$ and $N\in [n,\infty]$, define $Ric_N(cV):=c^2Ric_N(V)$. (Note that $Ric_N$ is $2$-homogeneous.)
\end{defn}
We say that $Ric_N\geq K$ for some $K\in \R$ if $Ric_N(V)\geq KF(V)^2$ for all $V\in TM$. It is proved by Ohta \cite{Oht09} that the bound $Ric_N(V)\geq KF(V)^2$ is equivalent to Lott-Villani and Sturm's curvature-dimension condition, which has many interesting applications (see \cite{Oht09,OS09}).

\subsection{Gradient and Laplacian}

Given a Finsler structure $F$ on a manifold $M$, its dual Finsler structure $F^*$ on the cotangent bundle $T^*M$ is defined by
\begin{equation*}
F^*(x,\xi):=\sup_{Y\in T_xM\setminus 0}\f{\xi(Y)}{F(x,Y)},\forall \xi\in T^*_xM.
\end{equation*}
The Legendre transformation $\cal{L}:TM\rightarrow T^*M$ is given by
\begin{equation*}
\cal{L}(Y):=\begin{cases}g_Y(Y,\cdot),&Y\neq 0,\\0,&Y=0.\end{cases}
\end{equation*}
It is easy to check that $\cal{L}$ is a diffeomorphism from $TM\setminus 0$ onto $T^*M\setminus 0$ and $F(Y)=F^*(\cal{L}(Y)),\forall Y\in TM$. Moreover, there holds the Cauchy-Schwarz inequality
\begin{equation*}
g_Y(Y,Z)\leq F(Y)F(Z),
\end{equation*}
for $\forall Y\neq 0,Z\in TM$.

Now for a smooth function $u:M\rightarrow \R$, we define the gradient vector $\n u(x)$ as $\n u(x):=\cal{L}^{-1}(du(x))\in T_xM$. In a local coordinate system, we can write it as
\begin{equation*}
\n u(x)=\begin{cases}g^{ij}(x,\n u)\dfrac{\pt u}{\pt x^i}\dfrac{\pt}{\pt x^j},& \text{ if }du(x)\neq 0,\\0, &\text{ otherwise},\end{cases}
\end{equation*}
where $g^{ij}(x,\n u)$ is the inverse of $g_{ij}(x,\n u)$. Also note that $g^{ij}(x,\n u)=g^{*ij}(x,du)$.

For a differentiable vector field $V$ on $M$ and $x\in M_V:=\{x|V(x)\neq 0\}$, we define $\n V\in T_x^*M\otimes T_xM$ by
\begin{equation*}
\n V(Y):=D^V_YV\in T_xM,\:Y\in T_xM.
\end{equation*}
Then the Hessian of $u$ is given by $\n^2 u:=\n(\n u)$ on $M_{\nabla u}$, which can also be seen as in $T^*_xM\otimes T^*_xM$ via
\begin{equation*}
\n^2 u(X,Y)=g_{\n u}(D_X^{\n u} \n u,Y).
\end{equation*}
One can check that $\n^2 u(X,Y)$ is symmetric. See details in \cite{OS14} or \cite{WX13}.

Now for a given positive $C^\infty$-measure $m$ on $M$, define the divergence of a differentiable vector field $V$ with respect to $m$ in the weak form by
\begin{equation*}
\int_M \phi\, div_m  Vdm=-\int_M d\phi(V)dm
\end{equation*}
for $\forall \phi \in C^\infty_c(M)$. In a local coordinate system where $dm=\sigma(x)dx$,
\begin{equation*}
div_mV=\f{1}{\sigma(x)}\f{\pt}{\pt x^i}\left(\sigma(x)V^i\right).
\end{equation*}
The Finsler Laplacian of $u\in W^{1,2}_{loc}(M)$ is given by $\Delta_m u=div_m (\n u)$. Recall the relationship between $\Delta_m u$ and $\n^2 u$ is that (see e.g. \cite[Lemma 3.3]{WX07})
\begin{equation*}
\Delta_m u=\sum_{i=1}^n \n^2 u(e_i,e_i)-S(\n u),\text{ on }M_{\nabla u}.
\end{equation*}
Here $\{e_i\}_{i=1}^n$ is the $g_{\n u}$-orthonormal basis and $S:TM\rightarrow \R$ is the $S$-curvature \cite{She01} given by
\begin{equation*}
S(V)=\f{d}{dt}\bigg|_{t=0}\Psi\circ \gamma(t),
\end{equation*}
where $\gamma$ is a geodesic with $\gamma'(0)=V$ and $dm=e^{-\Psi}vol_{\dot{\gamma}}$. Note that $S(cV)=cS(V)$, for $c>0$.

On the other hand, for a given smooth non-vanishing vector field $V$, we can define the weighted gradient vector and the weighted Laplacian on the weighted Riemannian manifold $(M,g_V,m)$ by
\begin{equation*}
\n^Vu(x)=g^{ij}(x,V)\f{\pt u}{\pt x^j}\f{\pt}{\pt x^i},
\end{equation*}
and
\begin{equation*}
\Delta_m^V u=div_m(\n^V u),
\end{equation*}
respectively. It is worth mentioning that $\n^{\n u}u=\n u$ and $\Delta_m^{\n u} u=\Delta_m u$ on $M_{\nabla u}$.


\subsection{The first and second variation formulas for arclength of geodesic}

Let $\gamma_0:[0,l]\rightarrow M$ be a unit speed geodesic in $M$. Suppose $\gamma(\e,s)$ is any variation of $\gamma_0(s)$ with $\e\in (-\e_0,\e_0)$. Then the first variation formula for arclength is (see e.g. \cite{Oht09})
\begin{equation*}
\f{\pt}{\pt \e}\bigg|_{\e=0}L(\gamma(\e,\cdot))=g_{\gamma_s}(\gamma_s,\gamma_\e)\big|_0^l,
\end{equation*}
where $\gamma_s$ is the unit tangent vector of $\gamma_0$ and $\gamma_\e=\f{\pt}{\pt \e}\gamma$ is the variational vector field.

Furthermore, the second variation formula is given by (see e.g. \cite{Oht09})
\begin{align*}
\f{\pt^2}{\pt \e^2}\bigg|_{\e=0}L(\gamma(\e,\cdot))&=\int_0^l \{g_{\gamma_s}(D^{\gamma_s}_{\gamma_s}({\gamma_\e}^\perp),D^{\gamma_s}_{\gamma_s}({\gamma_\e}^\perp))-g_{\gamma_s}(R^{\gamma_s}(\gamma_s,\gamma_\e)\gamma_\e,\gamma_s)\}ds\\
&+g_{\gamma_s}(D^{\gamma_s}_{\gamma_\e}{\gamma_\e},{\gamma_s})\big|_0^l,
\end{align*}
where $\gamma_\e^\perp$ means the normal part of the variational vector.

\subsection{Euler-Lagrange equation of energy functional, structure conditions and notations}\label{sec6.5}

Let $(M^n,F,m)$ be a compact Finsler measure space without boundary. We consider the following energy functional
\begin{equation*}
\cal{E}(u)=\int_M \left(\f{1}{2}\Phi(F^2(\n u))-Q(u)\right)dm,
\end{equation*}
where $\Phi \in C^3(\R^+)$ with $\Phi(0)=0$ and $Q\in C^2(\R)$. The Euler-Lagrange equation for $\cal{E}$ is given by
\begin{equation}\label{eq6.0}
div_m(\Phi'(F^2(\nabla u)) \n u)+q(u)=0,
\end{equation}
where $q(u)=Q'(u)$. We always assume $\Phi$ satisfies the following structure conditions: $\Phi(0)=0$, and there exist $p>1$, $\tau\geq 0$ and $c_1,c_2>0$ such that for any $V,W\in TM\setminus 0$,
\begin{equation}\label{cond6.1}
c_1(\tau+F(V))^{p-2}\leq \Phi'(F^2(V))\leq c_2(\tau+F(V))^{p-2}
\end{equation}
and
\begin{equation}\label{cond6.2}
c_1(\tau+F(V))^{p-2}F^2(W)\leq a(V)(W,W)\leq c_2(\tau+F(V))^{p-2}F^2(W),
\end{equation}
where $a(V)=2\Phi''(F^2(V))\cal{L}(V)\otimes \cal{L}(V)+\Phi'(F^2(V))g_V$.

\begin{rem}
If we choose
\begin{equation*}
\Phi(s)=\f{2}{p}s^{\f{p}{2}},
\end{equation*}
we obtain the $p$-Finsler-Laplace operator, which is the model for our structure conditions.
\end{rem}

We also use the following notations:
\begin{align*}
K(s)&:=\Phi'(s)s-\f{1}{2}\Phi(s),\\
\Lambda(s)&:=2\Phi''(s)s+\Phi'(s).
\end{align*}
Note that $K'(s)=\f{1}{2}\Lambda(s)$. And taking in \eqref{cond6.2} $W=V$, we have $\Lambda(s)\geq c_1(\tau+s^{\frac{1}{2}})^{p-2}>0$, for $s>0$.

\subsection{The regularity of the solutions to the equation}

By Theorem 1 in \cite{Tol84} (see also \cite{DiB83} and \cite{Lew83}), any solution $u$ of \eqref{eq6.0} satisfies $u\in C^{1,\alpha}(M)$ and
\begin{equation*}
||u||_{C^{1,\alpha}(M)}\leq C(|u|_\infty,M).
\end{equation*}

Furthermore, on any domain $\Omega$ such that $\inf_{\bar{\Omega}} F(\n u)>0$, Equation \eqref{eq6.0} is uniformly elliptic in $\Omega$. Then Theorem 6.3 in \cite[Chap.4]{LU68} implies $u\in C^{2,\alpha}(\Omega)$.

\subsection{Modulus of continuity estimate}\label{sec6.7}

We shall prove the following modulus of continuity estimate.
\begin{thm}\label{thm6.1}
Let $(M^n,F,m)$ be a compact Finsler measure space with nonnegative weighted Ricci curvature $Ric_\infty$ and $u\in W^{1,p}(M)$ a solution of Equation \eqref{eq6.0}. Suppose the barrier function $\varphi:[a,b]\rightarrow [\inf u,\sup u]$ satisfies
\begin{align}
&\quad\qquad\varphi'>0,\label{cond6.3}\\
\f{d}{ds}&\left(\f{\Lambda((\varphi')^2)\varphi''+q(\varphi)}{\varphi'\Phi'((\varphi')^2)}\right)<0.\label{cond6.4}
\end{align}
Moreover let $\psi$ be the inverse of $\varphi$, i.e. $\psi(\varphi(s))=s$. Then we have
\begin{equation*}
\psi(u(y))-\psi(u(x))-d(x,y)\leq 0,\: \forall x,y\in M.
\end{equation*}
\end{thm}

\begin{rem}\label{rem-perturbation}
In the result above, by simple perturbation, we can replace Condition \eqref{cond6.4} by the equality
\begin{align*}
\Lambda((\varphi')^2)\varphi''+q(\varphi)=0.
\end{align*}
\end{rem}

By letting $y$ approach $x$, we get the following gradient bound.
\begin{corr}\label{corr6.1}
Under the conditions of Theorem \ref{thm6.1}, for every $x\in M$ we have
\begin{equation*}
F(\n u(x))\leq \varphi'(\psi(u(x))).
\end{equation*}
\end{corr}

\begin{proof}[Proof of Theorem \ref{thm6.1}]
The proof is by contradiction. Define
\begin{equation*}
Z(x,y)=\psi(u(y))-\psi(u(x))-d(x,y).
\end{equation*}
Assume otherwise
\begin{equation*}
\max_{x,y\in M} Z(x,y)=Z(x_0,y_0)=\e_0>0.
\end{equation*}
Obviously $x_0\neq y_0$ and for any smooth unit speed curve $\gamma:[0,l]\rightarrow M$
\begin{equation*}
\cal{Z}(\gamma):=\psi(u(\gamma(l)))-\psi(u(\gamma(0)))-L(\gamma)\leq Z(\gamma(0),\gamma(l))\leq \e_0,
\end{equation*}
with equality when $\gamma=\gamma_0$, a length-minimising geodesic from $x_0$ to $y_0$.

Let $\gamma(\e,s)$ be any variation of $\gamma_0(s)$. The first derivative condition yields
\begin{equation*}
0=\psi'(u(y_0))g_{\n u}( \n u(y_0),\gamma_\e(l))-\psi'(u(x_0))g_{\n u}( \n u(x_0),\gamma_\e(0))-g_{\gamma_s}(\gamma_s,\gamma_\e)\big|_0^l.
\end{equation*}
Since the variation is arbitrary, we have
\begin{align*}
\psi'(u(y_0)) \n u(y_0)&=T(l),\\
\psi'(u(x_0)) \n u(x_0)&=T(0),
\end{align*}
or equivalently
\begin{align*}
\n u(y_0)&=\varphi'(z_y)T(l),\\
\n u(x_0)&=\varphi'(z_x)T(0),
\end{align*}
where $T$ is the unit tangent vector of $\gamma$, $z_y=\psi(u(y_0))$ and $z_x=\psi(u(x_0))$.

Next we will obtain some information from the following second derivative condition:
\begin{align*}
&0\geq\psi''(u)g_{\n u}( \n u,\gamma_\e(l))^2+\psi'(u)\left(g_{\n u}(D^{\n u}_{\gamma_\e}\n u,\gamma_\e(l))+g_{\n u}(\n u,D^{\n u}_{\gamma_\e}\gamma_\e(l))\right)\\
&-\psi''(u)g_{\n u}( \n u,\gamma_\e(0))^2-\psi'(u)\left(g_{\n u}(D^{\n u}_{\gamma_\e}\n u,\gamma_\e(0))+g_{\n u}(\n u,D^{\n u}_{\gamma_\e}\gamma_\e(0))\right)\\
&-\int_0^l \{g_{\gamma_s}(D^{\gamma_s}_{\gamma_s}({\gamma_\e}^\perp),D^{\gamma_s}_{\gamma_s}({\gamma_\e}^\perp))-g_{\gamma_s}(R^{\gamma_s}(\gamma_s,\gamma_\e)\gamma_\e,\gamma_s)\}ds-g_{\gamma_s}(D^{\gamma_s}_{\gamma_\e}{\gamma_\e},{\gamma_s})\big|_0^l,
\end{align*}
where we have suppressed some notations.

Choose at $x_0$ an orthonormal basis $\{e_1,\cdots,e_{n-1},e_n=T(0)\}$ with respect to $g_{\gamma_s(0)}$. Then parallel transport along $\gamma_0$ to produce an orthonormal basis $\{e_i(s)\}$ for each tangent space $T_{\gamma_0(s)}M$ (So $D^{\gamma_s}_{\gamma_s}e(s)=0$). Note that $e_n(s)=T(s)$ for each $s$. Then we consider the following three variations.

(1) Vary $y_0$. Choose the variation
\begin{equation*}
\gamma(\e,s)=\gamma_0(s+\e \f{s}{l}).
\end{equation*}
So $\gamma_\e(l)=e_n$ and $\gamma_\e(0)=0$. Then we get
\begin{align}
0&\geq\psi''(u)g_{\n u}( \n u,\gamma_\e(l))^2+\psi'(u)g_{\n u}(D^{\n u}_{\gamma_\e}\n u,\gamma_\e(l))\nonumber\\
 &= \psi''u_n(y_0)^2+\psi'u_{nn}(y_0)\nonumber\\
 &=\f{u_{nn}(y_0)-\varphi''(z_y)}{\varphi'(z_y)}.\label{eq6.1}
\end{align}

(2) Vary $x_0$. Choose the variation
\begin{equation*}
\gamma(\e,s)=\gamma_0(s+\e\f{l-s}{l}).
\end{equation*}
So $\gamma_\e(l)=0$ and $\gamma_\e(0)=e_n$. Then similarly we get
\begin{align}
0&\geq -\f{u_{nn}(x_0)-\varphi''(z_x)}{\varphi'(z_x)}.\label{eq6.2}
\end{align}

(3) Vary $\gamma_0$ along $e_i(s)$ for fixed $i<n$. Choose
\begin{equation*}
\gamma(\e,s)=\exp_{\gamma_0(s)}(\e e_i(s)).
\end{equation*}
So $\gamma_\e(s)=e_i(s)$. Therefore
\begin{align}
0&\geq \psi'(u(y_0))u_{ii}(y_0)-\psi'(u(x_0))u_{ii}(x_0)+\int_0^l g_{\gamma_s}(R^{\gamma_s}(\gamma_s,e_i)e_i,\gamma_s)ds\nonumber \\
 &=\f{u_{ii}(y_0)}{\varphi'(z_y)}-\f{u_{ii}(x_0)}{\varphi'(z_x)}+\int_0^l g_{\gamma_s}(R^{\gamma_s}(\gamma_s,e_i)e_i,\gamma_s)ds.\label{comparison1}
\end{align}
Then after summation from $i=1$ to $i=n-1$ and noting
\begin{equation*}
Ric_\infty(\gamma_s)=Ric(\gamma_s)+(\Psi\circ \gamma)''\geq 0,
\end{equation*}
we have
\begin{align}
0&\geq \f{\sum_{i<n}u_{ii}(y_0)}{\varphi'(z_y)}-\f{\sum_{i<n}u_{ii}(x_0)}{\varphi'(z_x)}-\int_0^l (\Psi\circ \gamma)''ds\nonumber\\
 &=\f{\sum_{i<n}u_{ii}(y_0)}{\varphi'(z_y)}-\f{\sum_{i<n}u_{ii}(x_0)}{\varphi'(z_x)}-(\Psi\circ \gamma)'\big|_0^l.\label{eq6.3}
\end{align}

Now recall Equation \eqref{eq6.0} is
\begin{align*}
0&=\Phi''(F^2(\n u))2u_{ij}u_iu_j+\Phi'(F^2(\n u))\Delta_m u+q(u)\\
 &=\Phi''(F^2(\n u))2u_{ij}u_iu_j+\Phi'(F^2(\n u))\left(\sum_{i=1}^n u_{ii}-S(\n u)\right)+q(u)
\end{align*}
on $M_{\nabla u}$. In particular, at $x_0$ or $y_0$
\begin{equation*}
\Phi''((\varphi')^2)2u_{nn}(\varphi')^2+\Phi'((\varphi')^2)\left(\sum_{i=1}^n u_{ii}-\varphi'(\Psi\circ \gamma)'\right)+q(\varphi)=0.
\end{equation*}
As a result we can solve
\begin{equation}\label{eq6.4}
\sum_{i<n}u_{ii}=-\f{\Lambda((\varphi')^2)u_{nn}+q(\varphi)}{\Phi'((\varphi')^2)}+\varphi'(\Psi\circ \gamma)'
\end{equation}
at $x_0$ or $y_0$.

Plugging $\eqref{eq6.4}$ into $\eqref{eq6.3}$ and using $\eqref{eq6.1}$ and $\eqref{eq6.2}$, finally we have
\begin{equation*}
0\geq -\f{\Lambda((\varphi')^2)\varphi''+q(\varphi)}{\varphi'\Phi'((\varphi')^2)}\bigg|_{z_y}+\f{\Lambda((\varphi')^2)\varphi''+q(\varphi)}{\varphi'\Phi'((\varphi')^2)}\bigg|_{z_x}.
\end{equation*}
Now taking \eqref{cond6.4} into account, since $z_y=z_x+d(x_0,y_0)+\e_0>z_x$, we get a contradiction. Then we must have
\begin{equation*}
Z(x,y)=\psi(u(y))-\psi(u(x))-d(x,y)\leq 0,
\end{equation*}
which is the desired result.
\end{proof}

\section{Gradient estimates and rigidity results for Finsler manifolds with nonnegative weighted Ricci curvature}\label{sec7}

In this section we continue the study on Finsler manifolds in Section \ref{sec6}. More precisely, we derive the gradient estimates of Modica type and some standard rigidity results.

Define
\begin{equation*}
c_u:=\sup_{r\in [\inf u,\sup u]} Q(r).
\end{equation*}

Then we shall prove the following gradient estimates of Modica type.
\begin{thm}\label{thm7.1}
Let $(M^n,F,m)$ be a compact Finsler measure space with nonnegative weighted Ricci curvature $Ric_\infty$ and $u\in W^{1,p}(M)$ a solution of Equation \eqref{eq6.0}. Then for all $x\in M$, there holds
\begin{equation}\label{eq2}
\Phi'(F^2(\n u))F^2(\n u)-\f{1}{2}\Phi(F^2(\n u))\leq c_u-Q(u).
\end{equation}
\end{thm}

\begin{proof}[Proof of Theorem \ref{thm7.1}]
Fix any
\begin{equation*}
c>c_u=\sup_{r\in [\inf u,\sup u]}Q(r).
\end{equation*}
Then we can solve
\begin{equation}\label{eq3.5}
K((\varphi')^2)=c-Q(\varphi)
\end{equation}
to get a solution $\varphi_c$ with $\varphi_c'>0$ and its image being $[\inf u,\sup u]$. In fact,
\begin{equation}\label{eq7.1}
s=s_0+\int_{\inf u}^{\varphi_c}\f{d\varphi}{\sqrt{K^{-1}\circ (c-Q(\varphi))}}
\end{equation}
for $\varphi_c\in [\inf u,\sup u]$. Differentiating \eqref{eq3.5} we know $\varphi_c$ solves
\begin{equation*}
\Lambda((\varphi')^2)\varphi''+q(\varphi)=0.
\end{equation*}
Now noting Remark \ref{rem-perturbation}, we can apply Corollary \ref{corr6.1} to get
\begin{equation*}
K(F^2(\n u))+Q(u)\leq K((\varphi_{c}')^2)+Q(\varphi_{c})= c.
\end{equation*}
Finally since $c>c_u$ is arbitrary we have
\begin{equation*}
K(F^2(\n u))+Q(u)\leq c_u.
\end{equation*}
So we complete the proof of Theorem \ref{thm7.1}.

\end{proof}

Another application of our modulus of continuity estimate, Theorem \ref{thm6.1}, is a rigidity result for $u$ concerning $c_u$ as follows.
\begin{thm}\label{thm4}
Let $u$ be as in Theorem \ref{thm7.1}. Suppose $\tau=0$ in the structure conditions \eqref{cond6.1} and \eqref{cond6.2}. Moreover, when $p>2$, we assume at any $r_0$ with $Q(r_0)=c_u$ and $Q'(r_0)=0$ there holds $Q(r)=Q(r_0)+O(|r-r_0|^p)$ as $r\rightarrow r_0$. If there exists a point $x_0\in M$ satisfying $Q(u(x_0))=c_u$ and $Q'(u(x_0))=0$, then $u$ is constant.
\end{thm}
\begin{rem}
Here we provide an example to illustrate some aspects of Theorem \ref{thm4}: Let $p=2$ and $Q(u)=\sin u$. Then Theorem \ref{thm4} indicates that the image $[\inf u,\sup u]$ of any non-constant solution $u$ can not contain points in $\{2k\pi+\f{\pi}{2}:k\in \N\}$, which gives a restriction on the solutions.
\end{rem}

\begin{proof}[Proof of Theorem \ref{thm4}]

Assume $u$ is not a constant function. Then $[\inf u,\sup u]$ has nonempty interior. Without loss of generality we may assume $u(x_0)>\inf u$. Then using \eqref{eq7.1} we get
\begin{equation*}
\psi_c(u(x_0))-\psi_c(\inf u)=\int_{\inf u}^{u(x_0)}\f{d\varphi}{\sqrt{K^{-1}\circ (c-Q(\varphi))}},\quad c>c_u.
\end{equation*}
Next we claim
\begin{equation*}
\lim_{c\rightarrow c_u^+}\int_{\inf u}^{u(x_0)}\f{d\varphi}{\sqrt{K^{-1}\circ (c-Q(\varphi))}}=+\infty.
\end{equation*}
So when $c$ is close enough to $c_u^+$, we get a contradiction to the modulus of continuity estimate
\begin{equation*}
\psi_c(u(y))-\psi_c(u(x))\leq d(x,y)<+\infty, \quad x,y\in M.
\end{equation*}

To prove the claim, first we observe that
\begin{equation}\label{eq4.1}
s^\f{p}{2}\leq \f{2}{\e_0}K(s),0<\e_0<\f{2}{p}c_1.
\end{equation}
In fact, define $G(s)=2K(s)-\e_0 s^{\f{p}{2}}$ with $\e_0<\f{2}{p}c_1$. Then $G(0)=0$ and
\begin{equation*}
G'(s)=\Lambda(s)-\e_0\f{p}{2}s^{\f{p}{2}-1}>0
\end{equation*}
by the assumption. Thus $G(s)\geq 0$, which is \eqref{eq4.1}.

Therefore, we obtain
\begin{equation*}
\sqrt{K^{-1}\circ (c_u-Q(\varphi))}\leq c (c_u-Q(\varphi))^{1/p}=c (Q(u(x_0))-Q(\varphi))^{1/p}.
\end{equation*}
Note that for $1<p\leq 2$, by Taylor expansion $Q(r)-Q(r_0)=O(|r-r_0|^2)=O(|r-r_0|^p)$. For $p>2$, by the assumption $Q(r)-Q(r_0)=O(|r-r_0|^p)$. In either case, we can conclude
\begin{equation*}
\sqrt{K^{-1}\circ (c_u-Q(\varphi))}\leq \tilde{c}(u(x_0)-\varphi),\quad \varphi\leq u(x_0),
\end{equation*}
which implies that
\begin{equation*}
\int_{\inf u}^{u(x_0)}\f{d\varphi}{\sqrt{K^{-1}\circ (c_u-Q(\varphi))}}=+\infty.
\end{equation*}
So we obtain the desired claim and finish the proof of Theorem \ref{thm4}.

\end{proof}

Also we can give a characterization of $c_u$.
\begin{thm}\label{thm5}
Under the same assumptions as Theorem \ref{thm4}, we have
\begin{equation*}
c_u=\max\{Q(\inf u),Q(\sup u)\}.
\end{equation*}
Moreover, if there exists a point $x_0\in M$ satisfying $Q(u(x_0))=c_u$, then either $u(x_0)=\inf u$ or $u(x_0)=\sup u$.
\end{thm}

\begin{proof}[Proof of Theorem \ref{thm5}]
Without loss of generality, we assume that $u$ is not a constant. Assume $c_u>\max\{Q(\inf u),Q(\sup u)\}$. Then there exists $r_0\in (\inf u,\sup u)$ such that
\begin{equation*}
\sup_{r\in [\inf u,\sup u]} Q(r)=c_u=Q(r_0).
\end{equation*}
So $r_0$ is a local maximum point for $Q$ and $Q'(r_0)=0$.

Meanwhile by the continuity of $u$, there exists a point $y_0$ such that $u(y_0)=r_0$. Thus $Q(u(y_0))=c_u$ and $Q'(u(y_0))=0$. So $u$ is constant by Theorem~\ref{thm4}, which is a contradiction. This completes the proof.

\end{proof}

\section{Noncompact manifolds and manifolds with boundary}\label{sec8}

In this section we consider various extensions of the two-point function method. More precisely, for noncompact manifolds without boundary, we shall make use of the ``translation invariance'' of the equation; while for compact manifolds with boundary and solutions with Dirichlet boundary condition, we derive a sharp barrier estimate near the boundary. In addition, we also consider anisotropic PDEs on certain unbounded domains with boundary in $\R^n$, which is a special Finsler measure space with boundary.

\subsection{Noncompact Riemannian manifolds without boundary}\label{sec:noncompact}

In this subsection we shall use the Cheeger-Gromov convergence of a sequence of complete pointed Riemannian manifolds $(N_k^n,g_k,x_k)$ to complete pointed Riemannian manifold $(N^n_\infty,g_\infty,x_\infty)$. By definition, it means that there exist an exhaustion $\{U_k\}$ of $N_\infty$ by open domains with $x_\infty\in U_k$ ($k\in \N$), and a sequence of diffeomorphisms $\Phi_k:U_k\rightarrow V_k:=\Phi_k(U_k)\subset N_k$ with $\Phi_k(x_\infty)=x_k$ such that $(U_k,\Phi^*_k(g_k|_{V_k}))$ converges smoothly and locally uniformly to $(N_\infty,g_\infty)$. Recall the Cheeger-Gromov convergence theorem (see e.g. \cite{CCG08}): any sequence of complete pointed Riemannian manifolds $(N_k^n,g_k,x_k)$ with uniformly bounded geometry converges to a complete pointed Riemannian manifold $(N^n_\infty,g_\infty,x_\infty)$ in the sense of Cheeger-Gromov, up to a subsequence.

Our first result in this subsection is as follows.
\begin{thm}\label{thm8.1}
Let $(M^n,g)$ be a complete noncompact Riemannian manifold. Assume $M$ is of bounded geometry and has nonnegative Ricci curvature. Let $u$ be a bounded viscosity solution of Equation \eqref{eq-G} on $M$. Suppose $\varphi:[a,b]\rightarrow [\inf u,\sup u]$ is a $C^2$ solution of
\begin{align}
\alpha(\varphi,\varphi')\varphi''+q(\varphi,\varphi')&=0\quad\text{on}\ [a,b];\label{cond8.2}\\
\varphi(a)=\inf u;\quad \varphi(b)=\sup u;\quad \varphi'&>0\quad\text{on}\ [a,b].\label{cond8.1}
\end{align}
Moreover let $\psi$ be the inverse of $\varphi$, i.e. $\psi(\varphi(z))=z$. Then we have
\begin{equation*}
\psi(u(y))-\psi(u(x))-d(x,y)\leq 0,\: \forall x,y\in M.
\end{equation*}
\end{thm}
\begin{rem}
It can be seen from the proof of Theorem \ref{thm8.1} that for noncompact manifolds the argument depends crucially on certain ``translation invariance'' of Equation \eqref{eq-G}. The exploitation of the translation invariance has already appeared in the known works which use $P$-function method. Our proof of Theorem \ref{thm8.1} is a combination of the translation invariance and the two-point functions method, and Theorem \ref{thm8.1} can recover the results in e.g. \cite{CGS94,Mod85}.
\end{rem}

\begin{proof}[Proof of Theorem \ref{thm8.1}]
The proof is by contradiction, and proceeds as in the proof of Theorem \ref{thm3.1} with suitable modifications. First by perturbation we may assume $\varphi$ satisfies
\begin{align*}
\f{d}{dz}&\left(\frac{q(\varphi,\varphi')+\varphi''\alpha(\varphi,\varphi')}{\varphi'\beta(\varphi,\varphi')}\right)<0.
\end{align*}

Let $\mathcal{N}$ denote the set of complete noncompact Riemannian manifolds $(N^n,g_N)$ with $Ric_N\geq 0$, and with bounded geometry of the same bounds as those for $ M$. For any $(N^n,g_N)\in \mathcal{N}$, let $S_N$ be the set of all viscosity solutions of Equation \eqref{eq-G} on $N$ with the same upper and lower bounds as those for $u$ on $M$. Set
\begin{equation*}
Z(N, v,x,y)=\psi(v(y))-\psi(v(x))-d_{g_N}(x,y), N\in \mathcal{N}, v\in S_N, x,y\in N.
\end{equation*}
It suffices to prove $\sup_{N\in \mathcal{N}, v\in S_N, x,y\in N}Z(N,v,x,y)\leq 0$. Assume otherwise
\begin{equation}\label{eq-8.1}
\sup_{N\in \mathcal{N}, v\in S_N, x,y\in N}Z(N,v,x,y)=\varepsilon_0>0.
\end{equation}
So there exist $N_k\in \mathcal{N}$, $v_k\in S_{N_k}$ and $x_k,y_k\in N_k$ ($k\in \N$) such that
\begin{equation}\label{eq-8.2}
\varepsilon_0-\frac{1}{k}\leq Z(N_k,v_k,x_k,y_k)\leq  \varepsilon_0.
\end{equation}

Now by Cheeger-Gromov convergence, we know
\begin{equation*}
(N_k,g_{N_k},x_k)\rightarrow (N_\infty, g_\infty, x_\infty),\quad k\rightarrow \infty,
\end{equation*}
in the sense of Cheeger-Gromov. Moreover, the followings hold true: (i) $v_k$ on $N_k$ converge locally uniformly to a viscosity solution $v_\infty$ of Equation \eqref{eq-G} on $N_\infty$ with the same bounds as $u$ (see e.g. \cite[Lem.~6.1]{CIL92}), and (ii) up to a subsequence $y_k\rightarrow y_\infty \in N_\infty$ in the sense of Cheeger-Gromov as $k\rightarrow \infty$.

So $v_\infty \in S_{N_\infty}$. Then sending $k$ to $\infty$ in Equation \eqref{eq-8.2} and noting \eqref{eq-8.1}, we may derive
\begin{equation*}
Z(N_\infty,v_\infty,x,y)\leq  \varepsilon_0,\quad x,y\in N_\infty,
\end{equation*}
with equality at $(x,y)=(x_\infty,y_\infty)$.

Then the remaining argument is the same as in the proof of Theorem \ref{thm3.1}, which finishes the proof.
\end{proof}

Analogously, we are able to prove the following result for the case where the Ricci curvature has a negative lower bound. We omit its proof here.

\begin{thm}\label{thm8.2}
Let $(M^n,g)$ be a complete noncompact Riemannian manifold. Assume $M$ is of bounded geometry and its Ricci curvature satisfies $Ric\geq (n-1)\kappa$, $\kappa< 0$. Let $u$ be a bounded viscosity solution of Equation \eqref{eq-G} on $M$. Suppose the barrier $\varphi:[a,b]\rightarrow [\inf u,\sup u]$ satisfies
\begin{align*}
&\varphi'>0,\\
\frac{q(\varphi,\varphi')+\varphi''\alpha(\varphi,\varphi')}{\varphi'\beta(\varphi,\varphi')}\bigg|_z&+(n-1)\frac{\rho'}{\rho}=0,
\end{align*}
where $\rho:[a,b]\rightarrow \R^+$ satisfies $\rho''+\kappa \rho=0$ and $(\frac{\rho'}{\rho})'>0$. Moreover let $\psi$ be the inverse of $\varphi$, i.e. $\psi(\varphi(z))=z$. Then we have
\begin{equation*}
\psi(u(y))-\psi(u(x))-d(x,y)\leq 0,\: \forall x,y\in M.
\end{equation*}
\end{thm}

\subsection{Compact Riemannian manifolds with boundary}\label{sec:boundary}

For manifolds with boundary, we discuss the solutions with Dirichlet boundary condition. First we would like to prove:
\begin{thm}\label{thm8.3}
Let $(M^n,g)$ be a compact Riemannian manifold with nonnegative Ricci curvature and with mean convex boundary $\pt M$. Let $u$ be a viscosity solution of Equation \eqref{eq-G} on $M$ with Dirichlet boundary condition $u|_{\pt M}=u_0$ for some constant $u_0$. Suppose $\varphi:[a,b]\rightarrow [\inf u,\sup u]$ is a $C^2$ solution of
\begin{align*}
\alpha(\varphi,\varphi')\varphi''+q(\varphi,\varphi')&=0\quad\text{on}\ [a,b];\\
\varphi(a)=\inf u;\quad \varphi(b)=\sup u;\quad \varphi'&>0\quad\text{on}\ [a,b].
\end{align*}
Moreover let $\psi$ be the inverse of $\varphi$, i.e. $\psi(\varphi(z))=z$. Then we have
\begin{equation*}
\psi(u(y))-\psi(u(x))-\bar{d}(x,y)\leq 0,\: \forall x,y\in M.
\end{equation*}
\end{thm}
\begin{rem}
Here in Theorem \ref{thm8.3} and below in Theorem \ref{thm8.4}, the generalized distance function $\bar{d}:M\times M\rightarrow \R$ is defined as
\begin{equation*}
\bar{d}(x,y):=\inf_\gamma Length(\gamma),
\end{equation*}
where the infimum is taken over all $C^1$ curves lying in $M$ and connecting $x$ and $y$.
\end{rem}
\begin{rem}
The papers \cite{FV10} and \cite{CFV12} considered the problem on, besides bounded domains, certain unbounded domains in $\R^n$ with mean convex boundary. By combining the arguments in Theorems \ref{thm8.1} and \ref{thm8.3}, we may see that the two-point functions method is applicable for the unbounded settings in \cite{FV10,CFV12}. Moreover, in \cite{FV10,CFV12} the solutions are assumed to be nonnegative; while here we do not need this additional assumption. See Theorem \ref{thm8.6} below for the general case.
\end{rem}

\begin{proof}[Proof of Theorem \ref{thm8.3}]
The proof is by contradiction, and proceeds as in the proof of Theorem \ref{thm3.1} with suitable modifications. First assume $\varphi(c)=u_0$ for $c\in [a,b]$. By perturbation we solve the following ODE:
\begin{align}
&\alpha(\varphi_\delta,\varphi_\delta')\varphi_\delta''+q(\varphi_\delta,\varphi_\delta')=-\delta (z-c)\cdot \varphi_\delta'\cdot \beta(\varphi_\delta,\varphi_\delta'),\label{eq-8.3}\\
&\qquad \qquad\varphi_\delta(a)=\varphi(a),\quad \varphi_\delta'(a)=\varphi'(a),\nonumber
\end{align}
where $\delta>0$ is chosen small. In the following we use $\varphi_\delta$ to derive a contradiction, and for simplicity we will use $\varphi$ to stand for $\varphi_\delta$.

As before define
\begin{equation*}
Z(x,y)=\psi(u(y))-\psi(u(x))-\bar{d}(x,y),  x,y\in M.
\end{equation*}
To get a contradiction let us assume
\begin{equation*}
\sup_{x,y\in M}Z(x,y)=\varepsilon_0>0.
\end{equation*}
Since $M$ is compact, there exist $x_0, y_0\in M$ such that
\begin{equation*}
Z(x_0,y_0)=\varepsilon_0>0.
\end{equation*}
Then there are three cases for the position of $(x_0,y_0)$.

\textbf{Case 1}: $x_0\in M$ and $y_0\in \pt M$. In this case it is easy to see that
\begin{equation*}
c-\psi(u(x))-d(x,\pt M)\leq \varepsilon_0,\quad x\in M,
\end{equation*}
with equality at $x=x_0$. The function $d(x,\pt M)$ may not be smooth. So we shall construct a smooth function $\tilde{d}$ on a small neighbourhood $U(x_0)$ of $x_0$ to replace it. The construction of such $\tilde{d}$ is standard (see e.g. \cite[pp.~73--74]{Wu79}), which may be stated as follows.

Note that $d(x_0,y_0)=d(x_0,\pt M):=l$. Let $\gamma$ be the unit speed length-minimizing geodesic joining $x_0$ and $y_0$ with $\gamma(0)=x_0$ and $\gamma(l)=y_0$. For any $X\in \exp_{x_0}^{-1}U(x_0)$, apply the parallel translate along $\gamma$ to $X$ to get $X(t)$ ($t\in [0,l]$) and decompose it as:
\begin{equation*}
X(t)=aX^\perp(t)+b\gamma'(t),
\end{equation*}
where $a$ and $b$ are constants along $\gamma$ satisfying $a^2+b^2=|X|^2$, and $X^\perp(t)$ is a parallel unit vector field along $\gamma$ orthogonal to $\gamma'(t)$.

Then we define the vector field
\begin{equation*}
W(t)=a  f(t)X^\perp(t)+b(1-\frac{t}{l}) \gamma'(t),
\end{equation*}
where $f:[0,l]\rightarrow \R^+$ is a $C^2$ function to be chosen. Next we can define the $n$-parameter family of curves $\gamma_{X}:[0,l]\rightarrow N$ ($X\in \exp_{x_0}^{-1}U(x_0)$) such that (1) $\gamma_0=\gamma$; (2) $\gamma_X(0)=\exp_{x_0}(W(0))$ and $\gamma_X(l)\in \pt M$; and (3) $W(t)$ is induced by the one-parameter family of curves $s\mapsto \gamma_{sX}(t)$ ($-s_0\leq s\leq s_0$, $0\leq t\leq l$); (4) $\gamma_X$ depends on $X$ smoothly. Let $\tilde{d}(x)$ be the length of the curve $\gamma_{X(x)}$ where $x=\exp_{x_0}(X)\in U(x_0)$. Then we get $\tilde{d}(x)\geq d(x,\pt M)$ on $U(x_0)$, $\tilde{d}(x_0)=l$, $D\tilde{d}|_{x_0} =-\gamma'(0)$, and
\begin{align*}
D^2\tilde{d}(X,X)&=-a^2\: f(l)^2II(X^\perp(l),X^\perp(l))+a^2\int_0^l ( (f')^2-f^2 K_M(X^\perp\wedge \gamma'))dt,
\end{align*}
where $II$ is the second fundamental form of the boundary at $y_0$ and $K_M$ denotes the sectional curvature of a two-plane.

Now we may analyse the inequality
\begin{equation*}
c-\psi(u(x))-\tilde{d}(x)\leq \varepsilon_0,\quad x\in U(x_0),
\end{equation*}
with equality at $x=x_0$, just as in the proof of Theorem \ref{thm3.1}. The details do not need elaboration and finally we will obtain (compare \eqref{comparison})
\begin{align*}
0&\leq -\frac{f^2(0)}{\varphi'\beta(\varphi,\varphi')}(q(\varphi,\varphi')+\varphi''\alpha(\varphi,\varphi'))\bigg|_{z_{x_0}}\\
&-f^2(l)H+(n-1)ff'\big|_0^l-\int_0^lf((n-1)f''+Ric(\gamma') f)ds.
\end{align*}
Choosing $f\equiv 1$, and using $H\geq 0$ and $Ric\geq 0$, we get
\begin{align*}
\frac{q(\varphi,\varphi')+\varphi''\alpha(\varphi,\varphi')}{\varphi'\beta(\varphi,\varphi')}\bigg|_{z_{x_0}}\leq 0.
\end{align*}
On the other hand, noting $z_{x_0}=c-l-\varepsilon_0<c$ and Equation \eqref{eq-8.3}, we have
\begin{align*}
\frac{q(\varphi,\varphi')+\varphi''\alpha(\varphi,\varphi')}{\varphi'\beta(\varphi,\varphi')}\bigg|_{z_{x_0}}> 0,
\end{align*}
a contradiction. So Case 1 is ruled out.

\textbf{Case 2}: $x_0\in \pt M$ and $y_0\in M$. In this case we still choose a unit speed length-minimizing geodesic $\gamma$ joining $x_0$ and $y_0$ with $\gamma(0)=x_0$ and $\gamma(l)=y_0$. Then we carry out the similar analysis as in Case 1 to get (again compare \eqref{comparison})
\begin{align*}
0&\leq \frac{f^2(l)}{\varphi'\beta(\varphi,\varphi')}(q(\varphi,\varphi')+\varphi''\alpha(\varphi,\varphi'))\bigg|_{z_{y_0}}\\
&-f^2(0)H+(n-1)ff'\big|_0^l-\int_0^lf((n-1)f''+Ric(\gamma') f)ds.
\end{align*}
Again taking $f\equiv 1$, and using $H\geq 0$ and $Ric\geq 0$, we shall get a contradiction to Equation \eqref{eq-8.3}. So Case 2 is also ruled out.

\textbf{Case 3}: $x_0\in M$ and $y_0\in M$. In this case we have
\begin{equation}\label{eq-8.4}
\psi(u(y))-\psi(u(x))-\bar{d}(x,y)\leq \varepsilon_0,
\end{equation}
for $x$ in a neighbourhood of $x_0$ and $y$ in a neighbourhood of $y_0$, with equality at $(x_0,y_0)$.

First we claim that $\bar{d}$ in the inequality above is indeed $d$, that is, $x_0$ and $y_0$ can be connected by a geodesic in the interior of $M$. To prove the claim, notice that the metric completion $\overline{M}$ of the Riemannian manifold $M$ with boundary is metrically complete  and locally compact. Then by Theorem~2.5.23 in \cite{BBI01}, we know $x_0$ and $y_0$ can be connected by a shortest path $\gamma_0$ in $\overline{M}$ such that $Length(\gamma_0)=\bar{d}(x_0,y_0)$. If $\gamma_0$ lies in the interior of $\overline{M}$, the claim follows immediately. Otherwise denote by $x_*$ (resp. $y_*$) the nearest point on $\gamma_0$ to $x_0$ (resp. $y_0$) which lies on $\pt M$. Then we have
\begin{equation*}
(\psi(u(y_0))-c-d(y_*,y_0))+(c-\psi(u(x_0))-d(x_0,x_*))=\varepsilon_0+\bar{d}(x_*,y_*).
\end{equation*}
Now without loss of generality we may assume $c-\psi(u(x_0))-d(x_0,x_*)=:\varepsilon_1$ is positive. Moreover in view of \eqref{eq-8.4} we may show
\begin{equation*}
c-\psi(u(x))-d(x,x_*)\leq \varepsilon_1,
\end{equation*}
or consequently
\begin{equation*}
c-\psi(u(x))-d(x,\pt M)\leq \varepsilon_1,
\end{equation*}
for $x$ in a neighbourhood of $x_0$, with equality at $x_0$. However this is impossible due to Case 1. So we have proved the claim.

With the claim at hand, the remaining is the same as in the proof of Theorem \ref{thm3.1}. Therefore Case 3 is ruled out, and we have finished the proof of Theorem \ref{thm8.3}.

\end{proof}

Similarly we can prove:

\begin{thm}\label{thm8.4}
Let $(M^n,g)$ be a compact Riemannian manifold with Ricci curvature $Ric\geq (n-1)\kappa$, $\kappa< 0$, and with horo-mean convex boundary such that $H_{\pt M}\geq (n-1)\sqrt{-\kappa}$. Let $u$ be a viscosity solution of Equation~\eqref{eq-G}  with Dirichlet boundary condition $u|_{\pt M}=u_0$ for some constant $u_0$. Suppose the barrier $\varphi:[a,b]\rightarrow [\inf u,\sup u]$ satisfies
\begin{align}
&\varphi'>0,\nonumber\\
\frac{q(\varphi,\varphi')+\varphi''\alpha(\varphi,\varphi')}{\varphi'\beta(\varphi,\varphi')}\bigg|_z&+(n-1)\frac{\rho'}{\rho}=0,\label{cond8.3}
\end{align}
where $\rho:[a,b]\rightarrow \R^+$ satisfies $\rho''+\kappa \rho=0$ and $(\frac{\rho'}{\rho})'>0$. Moreover let $\psi$ be the inverse of $\varphi$, i.e. $\psi(\varphi(z))=z$. Then we have
\begin{equation*}
\psi(u(y))-\psi(u(x))-\bar{d}(x,y)\leq 0,\: \forall x,y\in M.
\end{equation*}
\end{thm}
\begin{proof}[Sketch of the proof]
The proof goes through as in Theorem \ref{thm8.3}. Here we only point out some difference. We still have three cases. In Case 1, as before we will first get (compare \eqref{comparison})
\begin{align*}
0&\leq -\frac{f^2(0)}{\varphi'\beta(\varphi,\varphi')}(q(\varphi,\varphi')+\varphi''\alpha(\varphi,\varphi'))\bigg|_{z_{x_0}}\\
&-f^2(l)H+(n-1)ff'\big|_0^l-\int_0^lf((n-1)f''+Ric(\gamma') f)ds.
\end{align*}
Then let $f(s)=\rho(z_{x_0}+\varepsilon_0+s)$. By use of \eqref{cond8.3}, $(\frac{\rho'}{\rho})'>0$, $Ric(\gamma')\geq (n-1)\kappa$ and $H\geq (n-1)\sqrt{-\kappa}$, we can obtain
\begin{equation*}
0<\rho(c)\rho'(c)-\sqrt{-\kappa}\rho^2(c).
\end{equation*}
On the other hand, $\rho''+\kappa \rho=0$ and $(\frac{\rho'}{\rho})'>0$ implies
\begin{equation}\label{cond8.4}
-\sqrt{-\kappa}<\frac{\rho'}{\rho}<\sqrt{-\kappa}.
\end{equation}
So we have a contradiction.

In Case 2, we first get (compare \eqref{comparison})
\begin{align*}
0&\leq \frac{f^2(l)}{\varphi'\beta(\varphi,\varphi')}(q(\varphi,\varphi')+\varphi''\alpha(\varphi,\varphi'))\bigg|_{z_{y_0}}\\
&-f^2(0)H+(n-1)ff'\big|_0^l-\int_0^lf((n-1)f''+Ric(\gamma') f)ds.
\end{align*}
Let $f(s)=\rho(c+\varepsilon_0+s)$. Then similarly we will get
\begin{equation*}
0\leq -\rho(c+\e_0)\rho'(c+\e_0)-\sqrt{-\kappa}\rho^2(c+\e_0),
\end{equation*}
contradicting with \eqref{cond8.4}.

Lastly, Case 3 can be handled as in Theorem \ref{thm8.3}. So we complete the proof of Theorem \ref{thm8.4}.

\end{proof}

\begin{rem}
We have not treated the case of metrically complete manifolds with boundary here, but the techniques above (a combination of the methods of sections \ref{sec:noncompact} and \ref{sec:boundary}) easily extends to this case given a suitable version of the compactness theorem for Riemannian manifolds with boundary and with bounded geometry in a suitable sense.  Such a compactness result is available in \cite{Muller18}.  We discuss the Finsler case below in a somewhat more restricted setting (see Theorem \ref{thm8.6})
\end{rem}

\subsection{Finsler manifolds with boundary}

For a Finsler measure space $(M^n,F,m)$ with boundary $(\pt M, F|_{T\pt M}, m|_{\pt M})$, let $\nu_o$ and $\nu_i$ be the outer and inner unit normals of the boundary, respectively. Then the outer mean curvature $H^{out}$ and the inner mean curvature $H^{in}$ are given by
\begin{align*}
H^{out}&=-\sum_{k=1}^{n-1}g_{\nu_{o}}( D^{\nu_{o}}_{e_k}e_k,\nu_{o}),\\
H^{in}&=\sum_{k=1}^{n-1}g_{\nu_{i}}( D^{\nu_{i}}_{e_k}e_k,\nu_{i}),
\end{align*}
where $\{e_k\}_{k=1}^{n-1}$ are orthonormal basis of $T\pt M$ with respect to the corresponding metrics.

Now we may introduce the weighted outer and inner mean curvatures as follows.
\begin{defn}
The weighted outer mean curvature is defined by
\begin{equation*}
H^{out}_{\infty}=H^{out}-S(\nu_o),
\end{equation*}
where $S:TM\rightarrow \R$ is the S-curvature. Similarly, the weighted inner mean curvature is defined by
\begin{equation*}
H^{in}_{\infty}=H^{in}+S(\nu_i).
\end{equation*}

\end{defn}

Our first result in this subsection is concerned with compact Finsler measure spaces with boundary.

\begin{thm}\label{thm8.5}
Let $(M^n,F,m)$ be a compact Finsler measure space with nonnegative weighted Ricci curvature $Ric_\infty$ and with nonempty boundary. Moreover, assume both the outer and inner weighted mean curvatures of $(\pt M, F|_{T\pt M}, m|_{\pt M})$ are nonnegative. Let $u\in W^{1,p}(M)$ be a bounded solution of Equation \eqref{eq6.0} with Dirichlet boundary condition $u|_{\pt M}=u_0$ for some constant $u_0$. Suppose the barrier function $\varphi:[a,b]\rightarrow [\inf u,\sup u]$ satisfies
\begin{align*}
&\quad\qquad\varphi'>0,\\
&\Lambda((\varphi')^2)\varphi''+q(\varphi)=0.
\end{align*}
Moreover let $\psi$ be the inverse of $\varphi$, i.e. $\psi(\varphi(s))=s$. Then we have
\begin{equation*}
\psi(u(y))-\psi(u(x))-\bar{d}(x,y)\leq 0,\: \forall x,y\in M,
\end{equation*}
where $\bar{d}$ is the generalized distance function on $M\times M$.
\end{thm}

\begin{proof}[Sketch of the proof]
Assume $\varphi(c)=u_0$ for some $c\in [a,b]$. By perturbation we may assume that $\varphi$ satisfies
\begin{align*}
\f{\Lambda((\varphi')^2)\varphi''+q(\varphi)}{\varphi'\Phi'((\varphi')^2)}=-\delta (z-c)
\end{align*}
for some small $\delta>0$. Then the remaining argument proceeds as in the proof of Theorem \ref{thm8.3}, with suitable adaptation to the Finsler setting. Once again there are three cases, and we only need to consider the adaptation in Case~1 and Case~2.

In Case 1, i.e. $x_0\in M$ and $y_0\in \pt M$, by checking the proof of Theorem~\ref{thm6.1}, we first get (see \eqref{eq6.2})
\begin{align*}
u_{nn}(x_0)\geq \varphi''(z_{x_0}),
\end{align*}
and (compare \eqref{comparison1})
\begin{align*}
0&\geq -\f{\sum_{i=1}^{n-1}u_{ii}(x_0)}{\varphi'(z_{x_0})}+\int_0^l Ric(\gamma')ds+H^{out},
\end{align*}
where $H^{out}=-\sum_{k=1}^{n-1}g_{\nu_{o}}( D^{\nu_{o}}_{e_k}e_k,\nu_{o})$ is the outer mean curvature of $\pt M$ and $\nu_0=\gamma'(l)$ is the outer unit normal at $y_0$.

Taking into account
\begin{align*}
\frac{1}{\varphi'}\sum_{i<n}u_{ii}&=-\f{\Lambda((\varphi')^2)u_{nn}+q(\varphi)}{\varphi'\Phi'((\varphi')^2)}+(\Psi\circ \gamma)'(0), \text{ at } x_0,\\
Ric_\infty(\gamma')&=Ric(\gamma')+(\Psi\circ \gamma)''\geq 0,\\
H^{out}_{\infty}=H^{out}&-S(\nu_o)=H^{out}-(\Psi\circ \gamma)'(l)\geq 0,
\end{align*}
we can derive
\begin{align*}
\f{\Lambda((\varphi')^2)\varphi''+q(\varphi)}{\varphi'\Phi'((\varphi')^2)}\bigg|_{z_{x_0}}\leq 0.
\end{align*}
Then we will get a contradiction as before.

In Case 2, i.e. $x_0\in \pt M$ and $y_0\in M$, similarly we first get (see \eqref{eq6.1})
\begin{align*}
u_{nn}(y_0)\leq  \varphi''(z_{y_0}),
\end{align*}
and (compare \eqref{comparison1})
\begin{align*}
0&\geq \f{\sum_{i=1}^{n-1}u_{ii}(y_0)}{\varphi'(z_{y_0})}+\int_0^l Ric(\gamma')ds+H^{in},
\end{align*}
where $H^{in}=\sum_{k=1}^{n-1}g_{\nu_{i}}( D^{\nu_{i}}_{e_k}e_k,\nu_{i})$ is the inner mean curvature of $\pt M$ and $\nu_i=\gamma'(0)$ is the inner unit normal at $x_0$.

Again taking into account
\begin{align*}
\frac{1}{\varphi'}\sum_{i<n}u_{ii}&=-\f{\Lambda((\varphi')^2)u_{nn}+q(\varphi)}{\varphi'\Phi'((\varphi')^2)}+(\Psi\circ \gamma)'(l), \text{ at } y_0,\\
Ric_\infty(\gamma')&=Ric(\gamma')+(\Psi\circ \gamma)''\geq 0,\\
H^{in}_{\infty}=H^{in}&+S(\nu_i)=H^{in}+(\Psi\circ \gamma)'(0)\geq 0,
\end{align*}
we can derive
\begin{align*}
\f{\Lambda((\varphi')^2)\varphi''+q(\varphi)}{\varphi'\Phi'((\varphi')^2)}\bigg|_{z_{y_0}}\geq 0.
\end{align*}
Then we will get a contradiction as before.

This completes the proof.

\end{proof}

Lastly, we consider the anisotropic problem on $\R^n$ or the domains in $\R^n$. In \cite{FV14} and \cite{CFV14}, pointwise gradient bounds for solutions of anisotropic PDEs on $\R^n$ were obtained. More precisely, let us consider the general equation from \cite{CFV14} which is of the form
\begin{equation}\label{eq8.5}
div(B'(H(Du))DH(Du))+q(u)=0,
\end{equation}
where $B(t)=\Phi(t^2)/2$ in the language here, and $H:\R^n\rightarrow \R$ is a positively homogeneous function of degree $1$. Assume $B$, $H$ and $Q$ ($q=Q'$ and $Q$ is the $F$ in \cite{CFV14}) satisfy all the assumptions in \cite{CFV14}. It can be checked that the standard Euclidean space $\R^n$ equipped with such $H$ is equivalent to a Finsler measure space $(\R^n, F, m)$ with $m=dx$. In fact, $F$ and $H$ are dual to each other, i.e. $F=H^*$. And it is more appropriate to view $Du$ in \eqref{eq8.5} as an element in $T^*\R^n$; while $\nabla u=H(Du)DH(Du)$ an element in $T\R^n$. By this view, Equation \eqref{eq8.5} is the same as
\begin{equation}\label{eq8.6}
div_m(\Phi'(F^2(\nabla u)) \n u)+q(u)=0.
\end{equation}
Here we consider the anisotropic version of the results in \cite{FV10,CFV12}.
\begin{thm}\label{thm8.6}
Let $\Omega$ be a proper domain in $\R^n$ with $\pt \Omega\in C^{2,\alpha}_{loc}$ for some $\alpha\in (0,1)$. Assume the outer and inner anisotropic mean curvatures of $\pt \Omega$ are nonnegative. Moreover, suppose $\Omega$ is of either of the three forms: (i) $\Omega=\Omega_0\times \R^{n-k}$ where $\Omega_0\subset \R^k$ is a bounded domain and $1\leq k\leq n$; (ii) $\Omega$ is an epigraph, i.e. there exists $\Psi:\R^{n-1}\rightarrow \R$ such that $\Psi\in C^{2,\alpha}_{loc}(\R^{n-1})$, $||D \Psi||_{C^{1,\alpha}(\R^{n-1})}<+\infty$, and
\begin{equation*}
\Omega=\{x=(x',x_n):x_n>\Psi(x')\};
\end{equation*}
(iii) The (ordinary) second fundamental form $h$ of the boundary $\pt \Omega$ has a uniform bound $||h||_{C^\alpha(\pt \Omega)}\leq C$, and there exists $r_0>0$ such that $F:\pt \Omega \times [0,r_0]\rightarrow \Omega$ given by $F(x,s)=x-sN(x)$ is a diffeomorphism. Here $N(x)$ denotes the (ordinary) outer unit normal of $\pt \Omega$ at $x\in \pt \Omega$.

Let $u\in C^{1,\alpha}(\Omega)$ be a bounded weak solution of Equation \eqref{eq8.5} with Dirichlet boundary condition $u|_{\pt \Omega}=u_0$ for some constant $u_0$. Suppose the barrier function $\varphi:[a,b]\rightarrow [\inf u,\sup u]$ satisfies
\begin{align*}
&\quad\qquad\varphi'>0,\\
&\Lambda((\varphi')^2)\varphi''+q(\varphi)=0.
\end{align*}
Moreover let $\psi$ be the inverse of $\varphi$, i.e. $\psi(\varphi(s))=s$. Then we have
\begin{equation*}
\psi(u(y))-\psi(u(x))-H^*(y-x)\leq 0,\: \forall x,y\in \Omega.
\end{equation*}
\end{thm}
\begin{rem}
Although Case (iii) contains Case (i) and Case (ii), we separate Case (i) and Case (ii) as examples so that the readers can refer to the corresponding argument in \cite{FV10,CFV12}. Other special cases of Case (iii) include domains in $\R^n$ bounded by periodically rotationally symmetric hypersurfaces, as observed in \cite[p.~1985]{CFV12}.
\end{rem}
\begin{rem}
As mentioned above, Theorem \ref{thm8.6} is an extension of \cite{FV10,CFV12} to the anisotropic setting. For its proof we shall make use of two-point functions method together with the translation invariance of Equation \eqref{eq8.5}. In contrast, the $P$-function method seems not easy to apply in the setting of Theorem \ref{thm8.6}. Furthermore, as in the isotropic setting, here we also do not assume the solution has a sign. Besides, from the proof below, it is not hard to see that two-point functions method can apply to recover the results in \cite{FV14} and \cite{CFV14}.
\end{rem}

\begin{rem}
It can be checked that as a Finsler measure space, $(\R^n,H^*,dx)$ has zero weighted Ricci curvature $Ric_\infty$ and vanishing S-curvature. So the weighted anisotropic mean curvatures are equal to the non-weighted ones.
\end{rem}
\begin{proof}[Sketch of the proof]
As before assume $\varphi(c)=u_0$ for some $c\in [a,b]$. By perturbation we may suppose that $\varphi$ satisfies
\begin{align*}
\f{\Lambda((\varphi')^2)\varphi''+q(\varphi)}{\varphi'\Phi'((\varphi')^2)}=-\delta (z-c)
\end{align*}
for some small $\delta>0$. Note that $B$ and $\Phi$ are determined mutually, so that we are free to use either notation.

Here we only consider in details Case (ii) where $\Omega$ is an epigraph. Case~(i) is easier by following the lines in \cite{FV10,CFV12}; while Case~(iii) guarantees the compactness (in certain sense) of any sequence of domains in the class, so that the argument for Case~(ii) also works. We shall sketch the proof for Case~(iii) in the end.

Let $\mathcal{N}$ be the set of all domains
\begin{equation*}
\Omega_\theta=\{x=(x',x_n):x_n>\theta(x')\},
\end{equation*}
with $\theta\in C^{2,\alpha}_{loc}(\R^{n-1};\R)$ satisfying $||D \theta||_{C^{1,\alpha}(\R^{n-1})}\leq ||D \Psi||_{C^{1,\alpha}(\R^{n-1})}$, and with nonnegative outer and inner anisotropic mean curvatures along $\pt \Omega_\theta$.

For any $\Omega_\theta\in \mathcal{N}$, let $S_\theta$ be the set of all $C^{1,\alpha}$ weak solutions $v$ of \eqref{eq8.5} in $\Omega_\theta$ with $\inf u\leq v\leq \sup u$, $||Dv||_{C^\alpha(\Omega_\theta)}\leq ||Du||_{C^\alpha(\Omega)}$ and $v|_{\pt \Omega_\theta}=u_0$.

Set
\begin{equation*}
Z(\Omega_\theta,v,x,y):=\psi(v(y))-\psi(v(x))-H^*(y-x), \quad \Omega_\theta\in \mathcal{N}, \: v\in S_\theta,\: x,y\in \Omega_\theta.
\end{equation*}
As before, to get a contradiction let us assume
\begin{equation*}
\sup_{\Omega_\theta\in \mathcal{N}, \: v\in S_\theta,\: x,y\in \Omega_\theta}Z(\Omega_\theta,v,x,y)=\varepsilon_0>0.
\end{equation*}
So there exist $\Omega_{\theta_k}\in \mathcal{N}$, $v_k\in S_{\theta_k}$ and $x_k,y_k\in \Omega_{\theta_k}$ such that
\begin{equation*}
\varepsilon_0-\frac{1}{k}\leq  Z(\Omega_{\theta_k},v_k,x_k,y_k)\leq \varepsilon_0.
\end{equation*}
Let
\begin{equation*}
u_k(x):=v_k(x+x_k),\text{ and } \Omega_{\theta_k'}:=\Omega_{\theta_k}-\{x_k\},
\end{equation*}
for some $\theta'_k\in C^{2,\alpha}_{loc}(\R^{n-1};\R)$. So $u_k\in S_{\theta_k'}$ and we get
\begin{equation}\label{eq8.7}
\varepsilon_0-\frac{1}{k}\leq  Z(\Omega_{\theta_k'},u_k,0,y_k-x_k)\leq \varepsilon_0.
\end{equation}
Next we intend to take the limit in the inequality above (up to subsequence). For that purpose first note that $||u_k||_{C^{1,\alpha}(\Omega_{\theta_k'})}\leq C$. Then Lemma~6.37 of \cite{GT83} allows us to extend $u_k$ to $\tilde{u}_k\in C^{1,\alpha}(\R^n)$ such that
\begin{equation}\label{eq8.8}
\tilde{u}_k=u_k, \text{ in }\Omega_{\theta_k'},\quad ||\tilde{u}_k||_{C^{1,\alpha}(\R^n)}\leq C.
\end{equation}

Now we may discuss the limit in \eqref{eq8.7} as follows. Note that $\theta_k'(0)\leq 0$. So there are two cases.

\textbf{Case a}: $\{\theta_k'(0)\}_{k\in \N}$ is unbounded and so goes to $-\infty$. Then $\Omega_{\theta_k'}\rightarrow \R^n$ (noticing $||\nabla \theta_k'||_{C^{1,\alpha}(\R^{n-1})}\leq C$) and $u_k\rightarrow u_\infty$ in $C^{1,\alpha}_{loc}$, where $u_\infty$ is a $C^{1,\alpha}$ bounded weak solution of \eqref{eq8.5} on $\R^n$. Moreover, from \eqref{eq8.7} we may derive
\begin{equation*}
\psi(u_\infty(y))-\psi(u_\infty(x))-H^*(y-x)\leq \varepsilon_0,\quad x,y\in \R^n,
\end{equation*}
with equality at $x_0=0$ and some point $y_0$. Now interpret the problem in Finsler language, i.e. consider \eqref{eq8.6} on Finsler measure space $(\R^n, F,dx)$ instead of \eqref{eq8.5} on Euclidean space $\R^n$ with an anisotropic function $H$. Then we can derive a contradiction as in Theorem \ref{thm6.1}.

\textbf{Case b}: $\{\theta_k'(0)\}_{k\in \N}$ is bounded. Noting $||\nabla \theta_k'||_{C^{1,\alpha}(\R^{n-1})}\leq C$, we know $\theta_k'\rightarrow \theta_\infty$ in $C^{2,\alpha}_{loc}$, where $\theta_\infty \in C^{2,\alpha}_{loc}(\R^{n-1};\R)$. Also note that $\pt \Omega_{\theta_\infty}$ has nonnegative outer and inner anisotropic mean curvatures. On the other hand, in view of \eqref{eq8.8}, we know $\tilde{u}_k\rightarrow u_\infty$ in $C^{1,\alpha}_{loc}(\R^n)$, and it can be checked that $u_\infty$ is a $C^{1,\alpha}$ bounded weak solution of \eqref{eq8.5} in $\Omega_{\theta_\infty}$, $\inf u\leq u_\infty|_{\Omega_{\theta_\infty}}\leq \sup u$, and $u_\infty|_{\pt \Omega_{\theta_\infty}}=u_0$. Now once again from \eqref{eq8.7} we may derive
\begin{equation*}
\psi(u_\infty(y))-\psi(u_\infty(x))-H^*(y-x)\leq \varepsilon_0,\quad x,y\in \overline{\Omega_{\theta_\infty}},
\end{equation*}
with equality at $x_0=0$ and some point $y_0\in \overline{\Omega_{\theta_\infty}}$. Then interpreting the problem in Finsler language, we can obtain a contradiction as in Theorem~\ref{thm8.5}.

So we complete the proof for Case~(ii).

At last let us sketch the proof for Case~(iii) as promised. Let $\mathcal{N}$ be the set of all domains $\Omega'$ which satisfy the same condition (with the same bounds) as $\Omega$ in Case~(iii).

For any $\Omega'\in \mathcal{N}$, let $S_{\Omega'}$ be the set of all $C^{1,\alpha}$ weak solutions $v$ of \eqref{eq8.5} in $\Omega'$ with $\inf u\leq v\leq \sup u$, $||Dv||_{C^\alpha(\Omega')}\leq ||Du||_{C^\alpha(\Omega)}$ and $v|_{\pt \Omega'}=u_0$.

Set
\begin{equation*}
Z(\Omega',v,x,y):=\psi(v(y))-\psi(v(x))-H^*(y-x), \quad \Omega'\in \mathcal{N}, \: v\in S_{\Omega'},\: x,y\in \Omega'.
\end{equation*}
Then to get a contradiction, (after the translation) we can get a sequence of $(\Omega_k,u_k,x_k,y_k)$ with $\Omega_k\in \mathcal{N}$, $u_k\in S_{\Omega_k}$, $0\in \Omega_k$, and $y_k-x_k\in \Omega_k$, such that
\begin{equation}
\varepsilon_0-\frac{1}{k}\leq  Z(\Omega_k,u_k,0,y_k-x_k)\leq \varepsilon_0.
\end{equation}
Then there are two subcases: either $\Omega_k$ converges to $\R^n$, or $\Omega_k$ converges to some $\Omega_\infty \in \mathcal{N}$, up to a subsequence. This depends on whether $d(0,\pt \Omega_k)$ approaches $+\infty$ or not. In either subcases we can get a contradiction as in Case~(ii). So we can finish the proof for Case~(iii).

\end{proof}


\bibliographystyle{Plain}

\end{document}